\documentclass{amsart}
\usepackage{amsfonts, amsmath, amssymb}
\usepackage{stmaryrd}
\usepackage{graphicx}
\usepackage{psfrag}
\usepackage{xypic}
\usepackage{tableau}

\theoremstyle{plain}
\newtheorem{thm}{Theorem}[section]
\newtheorem{prop}[thm]{Proposition}
\newtheorem{cor}[thm]{Corollary}
\newtheorem{lemma}[thm]{Lemma}
\newtheorem{conj}[thm]{Conjecture}

\theoremstyle{remark}
\theoremstyle{definition}
\newtheorem{remark}[thm]{Remark}

\newtheorem{example}[thm]{Example}
\newtheorem{defn}[thm]{Definition}

\newcommand{\rank}{\operatorname{rank}}

\newcommand{\ot}{\leftarrow}

\newcommand{\Gr}{\operatorname{Gr}}
\newcommand{\Fl}{\operatorname{Fl}}

\newcommand{\pt}{\text{point}}

\newcommand{\bP}{{\mathbb P}}
\newcommand{\bG}{{\mathbb G}}
\newcommand{\bT}{{\mathbb T}}
\newcommand{\bk}{{\mathbb K}}
\newcommand{\N}{{\mathbb N}}
\newcommand{\Z}{{\mathbb Z}}
\newcommand{\Q}{{\mathbb Q}}

\newcommand{\cO}{{\mathcal O}}
\newcommand{\cA}{{\mathcal A}}
\newcommand{\cB}{{\mathcal B}}
\newcommand{\cV}{{\mathcal V}}
\newcommand{\cE}{{\mathcal E}}
\newcommand{\cF}{{\mathcal F}}

\newcommand{\cI}{{\mathcal I}}
\newcommand{\calL}{{\mathcal L}}
\newcommand{\cM}{{\mathcal M}}
\newcommand{\cQ}{{\mathcal Q}}
\newcommand{\cS}{{\mathcal S}}

\newcommand{\fm}{{\mathfrak m}}
\newcommand{\groth}{{\mathcal G}}

\newcommand{\bull}{{\scriptscriptstyle \bullet}}
\newcommand{\bOmega}{\boldsymbol{\Omega}}
\DeclareMathOperator{\Hom}{Hom} 
\DeclareMathOperator{\Tor}{Tor}
\DeclareMathOperator{\Spec}{Spec}
\DeclareMathOperator{\Sym}{Sym}

 \DeclareMathOperator{\codim}{codim}

\DeclareMathOperator{\GL}{GL} 
 \DeclareMathOperator{\Rep}{{\mathcal R}}

\DeclareMathOperator{\ch}{ch}

\newcommand{\id}{\operatorname{id}}

\newcommand{\wt}{\widetilde}

\newcommand{\wb}{\overline}

\newcommand{\bi}{{\mathbf i}}
\newcommand{\br}{{\mathbf r}}

\newcommand{\euler}[1]{\langle #1 \rangle}

\newcommand{\hmm}[1]{\mbox{}\hspace{#1mm}\mbox{}}
\newcommand{\vmm}[1]{\mbox{}\vspace{#1mm}\mbox{}}

\newenvironment{romenum}{\begin{enumerate}}{\end{enumerate}}

\newcommand{\toto}{\text{\raisebox{3pt}{$\ \displaystyle
      \mathop{\to}_{\displaystyle \to}\ $}}}

\newcommand{\pic}[2]{\includegraphics[scale=#1]{#2}}

\newcommand{\ignore}[1]{}

\psfrag{R}{$R$}
\psfrag{la}{$\lambda$}
\psfrag{mu}{$\mu$}

\begin{document}

\title{Quiver coefficients of Dynkin type}
\author{Anders~Skovsted~Buch}
\address{Department of Mathematics, Rutgers University, 110
Frelinghuysen Road, Piscataway, NJ 08854, USA}
\email{asbuch@math.rutgers.edu}
\thanks{The author was partially supported by NSF Grant DMS-0603822}
\date{March 17, 2007; Revised August 21, 2007}
\maketitle

\section{Introduction}

Let $Q = (Q_0,Q_1)$ be a quiver, consisting of a finite set of
vertices $Q_0$ and a finite set of arrows $Q_1$.  Each arrow $a \in
Q_1$ has a head $h(a)$ and a tail $t(a)$ in $Q_0$.  For convenience we
will assume that the vertex set is an integer interval, $Q_0 =
\{1,2,\dots,n\}$.  Let $e = (e_1,\dots,e_n) \in \N^n$ be a dimension
vector, and fix vector spaces $E_i = \bk^{e_i}$ for $i \in Q_0$ over a
field $\bk$.  The representations of $Q$ on these vector spaces form
the affine space $V = \bigoplus_{a \in Q_1} \Hom(E_{t(a)}, E_{h(a)})$,
which has a natural action of the group $\bG = \GL(E_1) \times \dots
\times \GL(E_n)$ given by $(g_1,\dots,g_n) . (\phi_a)_{a \in Q_1} =
(g_{h(a)}\, \phi_a\, g_{t(a)}^{-1})_{a \in Q_1}$.

Define a {\em quiver cycle\/} to be any $\bG$-stable closed
irreducible subvariety $\Omega$ in $V$.  A quiver cycle determines an
equivariant (Chow) cohomology class $[\Omega] \in H^*_\bG(V)$ and an
equivariant Grothendieck class $[\cO_\Omega] \in K_\bG(V)$.  These
classes are well understood when the quiver $Q$ is {\em
  equioriented\/} of type A, that is, a sequence $\{1 \to 2 \to \dots
\to n\}$ of arrows in the same direction.  In this case, a formula for
the cohomology class $[\Omega]$ was given in joint work with Fulton
\cite{buch.fulton:chern}, and this formula was generalized to
$K$-theory in \cite{buch:grothendieck}.  The $K$-theory formula states
that the Grothendieck class $[\cO_\Omega]$ is given by
\[ [\cO_\Omega] = \sum_\mu c_\mu(\Omega)\,
   \groth_{\mu_1}(E_2-E_1) \groth_{\mu_2}(E_3-E_2) \cdots
   \groth_{\mu_{n-1}}(E_n-E_{n-1}) \, \in K_\bG(V)
\]
where the sum is over finitely many sequences $\mu =
(\mu_1,\dots,\mu_{n-1})$ of partitions $\mu_i$.  Each factor
$\groth_{\mu_i}(E_{i+1} - E_i)$ is obtained by applying the stable
Grothendieck polynomial for $\mu_i$ to the standard representations of
$\bG$ on $E_{i+1}$ and $E_i$.  This notation will be explained in
section \ref{S:quivcoef}.

The coefficients $c_\mu(\Omega)$ are interesting geometric and
combinatorial invariants called (equioriented) {\em quiver
  coefficients}.  They are integers and are non-zero only when the sum
$\sum |\mu_i|$ of the weights of the partitions is greater than or
equal to the codimension of $\Omega$.  The coefficients for which this
sum equals $\codim(\Omega)$ describe the cohomology class of $\Omega$
and are called {\em cohomological quiver coefficients}.  It was proved
in \cite{knutson.miller.ea:four} that cohomological quiver
coefficients are non-negative, and in \cite{buch:alternating,
  miller:alternating} that the more general $K$-theoretic quiver
coefficients have {\em alternating signs}, in the sense that
$(-1)^{\sum |\mu_i| - \codim(\Omega)} c_\mu(\Omega)$ is a non-negative
integer.  These properties had earlier been conjectured in
\cite{buch.fulton:chern, buch:grothendieck}, and special cases had
been proved in \cite{buch:on, buch.kresch.ea:schubert,
  buch.kresch.ea:grothendieck}.  The equioriented quiver coefficients
can furthermore be expressed in terms of counting factor sequences
\cite{buch.fulton:chern, buch:on, knutson.miller.ea:four,
  buch:alternating, buch.kresch.ea:stable}.  They are known to
generalize Littlewood-Richardson coefficients
\cite{buch.fulton:chern}, ($K$-theoretic) Stanley coefficients
\cite{buch:stanley, buch:grothendieck}, and the monomial coefficients
of Schubert and Grothendieck polynomials
\cite{buch.kresch.ea:schubert, buch.kresch.ea:grothendieck}.  The
equioriented quiver coefficients are themselves special cases of the
$K$-theoretic Schubert structure constants on flag manifolds
\cite{lenart.robinson.ea:grothendieck, buch:alternating,
  buch.sottile.ea:quiver}.

The purpose of this paper is to introduce and study a more general
notion of quiver coefficients, which can be defined for an arbitrary
quiver $Q$ without oriented loops.  For each vertex $i \in Q_0$, we
define $M_i = \bigoplus_{a : h(a)=i} E_{t(a)}$ to be the direct sum of
all vertex vector spaces at the tails of arrows pointing to $i$.  (If
there are two or more arrows to $i$ from a vertex $j$, then $E_j$ is
included multiple times as a summand of $M_i$.)  Given a quiver cycle
$\Omega \subset V$, we show that there are unique coefficients
$c_\mu(\Omega) \in \Z$, indexed by sequences $\mu =
(\mu_1,\dots,\mu_n)$ of partitions such that the length $\ell(\mu_i)$
is at most $e_i$, for which
\begin{equation} \label{E:quivcoef_intro}
  [\cO_\Omega] = \sum_\mu c_\mu(\Omega) \, 
  \groth_{\mu_1}(E_1-M_1) \groth_{\mu_2}(E_2-M_2) \cdots
  \groth_{\mu_n}(E_n-M_n) \,.
\end{equation}
As in the equioriented case, a coefficient $c_\mu(\Omega)$ can be
non-zero only if $\sum |\mu_i| \geq \codim(\Omega)$, and the lowest
degree coefficients describe the cohomology class $[\Omega]$.
However, the defining linear combination (\ref{E:quivcoef_intro})
might possibly be infinite, which makes sense modulo the gamma
filtration on $K_\bG(V)$.  We pose the following.

\begin{conj} \label{C:altsign}
  Let $Q$ be a quiver without oriented loops and $\Omega \subset V$ a
  quiver cycle.

(a) Only finitely many of the quiver coefficients $c_\mu(\Omega)$ for
$\Omega$ are non-zero.  In other words, the sum
(\ref{E:quivcoef_intro}) is finite.

(b) All cohomological quiver coefficients $c_\mu(\Omega)$, with
$\sum|\mu_i| = \codim(\Omega)$, are non-negative.

(c) If $\Omega$ has rational singularities, then the quiver
coefficients for $\Omega$ have alternating signs, i.e.\ 
$(-1)^{\sum|\mu_i| - \codim(\Omega)} c_\mu(\Omega) \geq 0$.
\end{conj}

Our main result is a formula for the quiver coefficients when the
quiver $Q$ is of Dynkin type and $\Omega$ has rational singularities.
A quiver is of Dynkin type if the underlying (un-directed) graph is a
simply-laced Dynkin diagram, i.e.\ a disjoint union of Dynkin diagrams
of types A, D, and E.  In this case, every quiver cycle is an orbit
closure \cite{gabriel:unzerlegbare}.  Bobi\'nski and Zwara have proved
that all orbit closures have rational singularities if $Q$ is a quiver
of type A and $\bk$ is an algebraically closed field
\cite{bobinski.zwara:normality}, or if $Q$ is of type D and $\bk$ is
algebraically closed of characteristic zero
\cite{bobinski.zwara:schubert} (see also
\cite{lakshmibai.magyar:degeneracy} for the equioriented case).
Our formula relies on an explicit desingularization of an orbit
closure given by Reineke \cite{reineke:quivers}, as well as a list of
geometric and combinatorial properties of stable Grothendieck
polynomials established in \cite{buch:littlewood-richardson,
  buch:grothendieck}, and it proves the finiteness part (a) of
Conjecture~\ref{C:altsign}.  Our new formula generalizes the formula
for equioriented quiver coefficients proved in
\cite{buch:grothendieck}, but requires more operations on Grothendieck
polynomials, including multiplication and Grothendieck polynomials
indexed by sequences of negative integers.  For quivers of type A$_3$,
we prove the full statement of Conjecture~\ref{C:altsign}, and we
provide positive combinatorial formulas for the quiver coefficients in
terms of counting set-valued tableaux.

We remark that the positivity properties of quiver cycles suggested by
Conjecture~\ref{C:altsign} are analogous to positivity properties
satisfied by a closed and irreducible subvariety $Y$ of a homogeneous
space $G/P$.  In fact, the cohomology class of $Y$ can be uniquely
written as a positive linear combination of Schubert classes, where
the coefficients count the intersection points of $Y$ with the dual
Schubert varieties placed in general position.  Furthermore, Brion has
proved that if $Y$ has rational singularities, then the Grothendieck
class of $Y$ is an alternating linear combination of $K$-theoretic
Schubert classes \cite{brion:positivity}.  Aside from this analogy,
our conjecture is supported by computer experiments.

Some other formulas for quiver cycles of Dynkin type have been given,
which do not involve quiver coefficients.  First of all, Feh\'er and
Rim\'anyi have proved that the cohomology class of an orbit closure of
Dynkin type is uniquely determined, up to a constant, by the property
that its restriction to any disjoint orbit vanishes
\cite{feher.rimanyi:classes}.  Rim\'anyi and the author have used this
result to prove a positive combinatorial formula for the cohomology
class of any orbit closure for a (non-equioriented) quiver of type A,
which expresses this class as a sum of products of Schubert
polynomials \cite{buch.rimanyi:formula}.  A conjectured $K$-theory
version furthermore expresses the Grothendieck classes of such orbit
closures as alternating sums of products of Grothendieck polynomials.
These formulas generalize the (non-stable) component formulas for
equioriented quivers proved by Knutson, Miller, and Shimozono in
cohomology \cite{knutson.miller.ea:four} and by the author in
$K$-theory \cite{buch:alternating}.  Despite the positivity displayed
by the generalized component formulas, we have not been able to relate
them to positivity properties of quiver coefficients in the
non-equioriented cases.  Finally, a recent preprint of Knutson and
Shimozono \cite{knutson.shimozono:kempf} contains a formula for the
Grothendieck class of any orbit closure of Dynkin type which has
rational singularities.  This formula is stated in terms of Demazure
operators, but does not to our knowledge suggest any positivity
properties of quiver cycles.

This paper is organized as follows.  In section~\ref{S:groth} we
recall the definition and required properties of stable Grothendieck
polynomials.  Section~\ref{S:quivcoef} describes the equivariant
Grothendieck class of a quiver cycle, defines the corresponding quiver
coefficients, and discusses the available evidence for
Conjecture~\ref{C:altsign}.  We also give an example of an orbit
closure for which the associated quiver coefficients do not have
alternating signs.  This orbit closure was earlier studied by Zwara
\cite{zwara:orbit}, who proved that it does not have rational
singularities.  In section~\ref{S:degloci} we interpret quiver
coefficients in terms of formulas for degeneracy loci defined by a
quiver of vector bundles over a base variety.  In
section~\ref{S:resolution} we describe Reineke's desingularization of
orbit closures of Dynkin type.  This desingularization is used in
section~\ref{S:formula} to prove a combinatorial formula for quiver
coefficients of Dynkin type.  The last section contains the proof of
Conjecture~\ref{C:altsign} for quivers of type A$_3$.

Our formula for orbit closures of Dynkin type was proved at the time
the preprint \cite{knutson.shimozono:kempf} became available.  We do,
however, thank Allen Knutson for earlier suggesting that resolutions
that we used to compute quiver coefficients of types A and D might be
special cases of Reineke's general construction.  We have benefited
from many discussions with Rich\'ard Rim\'anyi on this general
subject, and from answers to questions and useful references provided
by Wilbert van der Kallen and Michel Brion regarding group actions and
equivariant $K$-theory.  We also thank Johan de Jong, Friedrich Knop,
Chris Woodward, and Bobi{\'n}ski Zwara for helpful comments and
answers to questions.


\section{Grothendieck polynomials}
\label{S:groth}

In this section we fix notation for stable Grothendieck polynomials
and state the required properties.  We refer to
\cite{buch:littlewood-richardson, buch:grothendieck} for more details.

A {\em partition\/} is a weakly decreasing sequence of non-negative
integers $\lambda = (\lambda_1 \geq \lambda_2 \geq \dots \geq
\lambda_\ell \geq 0)$.  
The {\em weight\/} of $\lambda$ is the sum $|\lambda| = \sum
\lambda_i$ of its parts and the {\em length\/} $\ell(\lambda)$ is the
number of non-zero parts.
We will identify the
partition $\lambda$ with its {\em Young diagram}, which has
$\lambda_1$ boxes in the top row, $\lambda_2$ boxes in the next row,
etc.  A {\em set-valued tableau\/} of shape $\lambda$ is a filling $T$
of the boxes of $\lambda$ with finite non-empty sets of positive
integers, such that the largest integer in any box is smaller than or
equal to the smallest integer in the box to the right of it, and
strictly smaller than the smallest integer in the box below it.  Given
an infinite set of commuting variables $x = (x_1, x_2, \dots)$, we let
$x^T$ denote the monomial in which the exponent of $x_i$ is the number
of boxes of $T$ containing $i$, and we let $|T|$ be the (total) degree
of $x^T$.  For example, the set-valued tableau
\[ T = \tableau{18}{
  {1\hmm{-.2},\hmm{-.3}2} & {2} & 
  {2\hmm{-.3},\hmm{-.8}5\hmm{-.4},\hmm{-.8}8} \\
  {4} & {7\hmm{-.2},\hmm{-.3}8}} 
\]
has shape $\lambda = (3,2)$ and gives $x^T = x_1 x_2^{\,3} x_4 x_5 x_7
x_8^{\,2}$ and $|T| = 9$.

The {\em single stable Grothendieck polynomial\/} for the partition
$\lambda$ is defined as the formal power series
\[ \groth_\lambda = \groth_\lambda(x) = \sum_T \, (-1)^{|T|-|\lambda|}
   \, x^T \,,
\]
where the sum is over all set-valued tableaux $T$ of shape $\lambda$.
This power series is symmetric, and its term of lowest degree is the
Schur function $s_\lambda$.  It was proved in
\cite{buch:littlewood-richardson} to be a special case of the stable
Grothendieck polynomials indexed by permutations of Fomin and Kirillov
\cite{fomin.kirillov:grothendieck}, which in turn were constructed as
limits of Lascoux and Sch\"utzenberger's ordinary Grothendieck
polynomials.  By convention, a stable Grothendieck polynomial applied
to a finite set of variables is defined by
$\groth_\lambda(x_1,\dots,x_p) =
\groth_\lambda(x_1,\dots,x_p,0,0,\dots)$.

Given a set-valued tableau $T$, define its {\em word\/} $w(T)$ to be
the sequence of integers in its boxes when read one row at the time
from left to right, with the rows ordered from bottom to top.
Integers in the same box are arranged in increasing order.  For
example, the tableau displayed above gives $w(T) =
(4,7,8,1,2,2,2,5,8)$.  A word of positive integers is called a {\em
  reverse lattice word\/} if every occurrence of an integer $i \geq 2$
is followed by more occurrences of $i-1$ than of $i$.  The {\em
  content\/} of a word is the sequence $\nu = (\nu_1,\nu_2,\dots)$
where $\nu_i$ is the number of occurrences of $i$ in the word.  For
any partition $\mu = (\mu_1,\dots,\mu_l)$, let $u(\mu) = (l^{\mu_l},
\dots, 2^{\mu_2}, 1^{\mu_1})$ be the word of the tableau of shape
$\mu$ in which all boxes in row $i$ contains the single integer $i$.
We need the following generalization of the classical
Littlewood-Richardson rule from
\cite[Thm.~5.4]{buch:littlewood-richardson} (an alternative proof can
be found in \cite[\S 3.5]{buch.kresch.ea:stable}).

\begin{thm} \label{T:mult}
The product of two stable Grothendieck polynomials is given by
\[ \groth_\lambda \cdot \groth_\mu = 
   \sum_\nu c^\nu_{\lambda \mu}\, \groth_\nu
\]
where the sum is over all partitions $\nu$, and $c^\nu_{\lambda\mu}$ is
equal to $(-1)^{|\nu|-|\lambda|-|\mu|}$ times the number of set-valued
tableaux $T$ of shape $\lambda$ for which the composition $w(T)
u(\mu)$ is a reverse lattice word with content $\nu$.
\end{thm}

For example, the set-valued tableaux
\raisebox{.5mm}{$\tableau{13}{{1}}$}\,,
\raisebox{.5mm}{$\tableau{13}{{2}}$}\,, and
\raisebox{.5mm}{$\tableau{13}{{1\hmm{-.2},\!2}}$}\, correspond to the
terms of the product $\groth_{\tableau{6}{{}}} \cdot
\groth_{\tableau{6}{{}}} = \groth_{\tableau{6}{{}&{}}} +
\groth_{\tableau{6}{{}\\{}}} - \groth_{\tableau{6}{{}&{}\\{}}}$.  If a
coefficient $c^\nu_{\lambda \mu}$ is non-zero, then $|\lambda|+|\mu|
\leq |\nu|$ and (the Young diagrams of) $\lambda$ and $\mu$ can be
contained in $\nu$.

Theorem~\ref{T:mult} implies that the linear span $\Gamma = \bigoplus \Z
\groth_\lambda$ of all stable Grothendieck polynomials is a
commutative ring.  The stable Grothendieck polynomials are linearly
independent since the term of lowest degree in $\groth_\lambda$ is the
Schur function $s_\lambda$.

If $\lambda$, $\mu$, and $\nu$ are partitions such that $\lambda$ and
$\mu$ fit inside a rectangular partition $R$, we define
\[ d^\nu_{\lambda \mu} = c^\rho_{R\,\nu} \ \ , \ \ \text{where } \ \ 
   \rho = (R+\mu,\lambda) = \ \raisebox{-7mm}{\pic{.40}{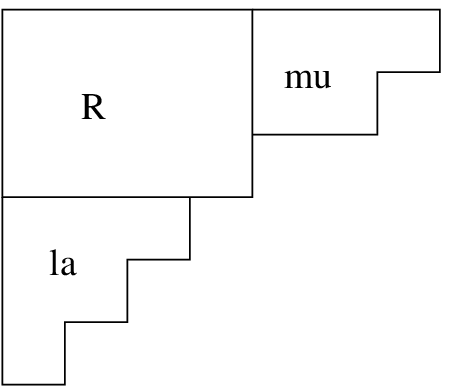}}
\]
is the partition obtained by attaching $\lambda$ and $\mu$ to the
bottom and right sides of $R$.  
This constant $d^\nu_{\lambda \mu}$ is independent of the choice of
rectangle $R$, and it is non-zero only if $|\nu| \leq |\lambda|+|\mu|$
and $\lambda, \mu \subset \nu$
\cite[Thm.~6.6]{buch:littlewood-richardson}.  These constants define a
coproduct $\Delta : \Gamma \to \Gamma \otimes \Gamma$ given by
$\Delta(\groth_\nu) = \sum_{\lambda,\mu} d^\nu_{\lambda \mu}\,
\groth_\lambda \otimes \groth_\mu$, which gives $\Gamma$ a structure
of commutative and cocommutative bialgebra with unit and counit
\cite[Cor.~6.7]{buch:littlewood-richardson}.

Given an additional set of commuting variables $y = (y_1,y_2,\dots)$,
define the {\em double stable Grothendieck polynomial\/} for the
partition $\nu$ by 
\begin{equation*} \groth_\nu(x;y) = \sum_{\lambda,\mu} d^\nu_{\lambda \mu}\,
   \groth_\lambda(x) \cdot \groth_{\mu'}(y) \,,
\end{equation*}
where $\mu'$ is the conjugate partition of $\mu$, obtained by
interchanging the rows and columns of $\mu$.  These power series are
separately symmetric in each set of variables $x$ and $y$, and they
satisfy the identities 
\begin{equation} \label{E:supersym}
  \groth_\nu(1-a^{-1},x \ ; \, 1-a,y) = \groth_\nu(x;y)
\end{equation}
for any indeterminate $a$ \cite{fomin.kirillov:grothendieck},
and
\begin{equation} \label{E:split}
  \groth_\nu(x,z; y,w) = \sum_{\lambda,\mu} d^\nu_{\lambda,\mu}\,
  \groth_\lambda(x;y) \, \groth_\mu(z;w)
\end{equation}
for arbitrary sets of variables $x$, $y$, $z$, and $w$
\cite[(6.1)]{buch:littlewood-richardson}.  Another useful identity is
the factorization formula \cite[Cor.~6.3]{buch:littlewood-richardson},
which states that
\begin{multline} \label{E:factor}
  \groth_{R+\mu,\lambda}(x_1,\dots,x_p; y_1,\dots,y_q) = \\
  \groth_\lambda(0;y_1,\dots,y_q) \cdot
  \groth_R(x_1,\dots,x_p;y_1,\dots,y_q) \cdot
  \groth_\mu(x_1,\dots,x_p)
\end{multline}
whenever $\lambda$ and $\mu$ are partitions with $\lambda_1 \leq q$
and $\ell(\mu) \leq p$, and $R = (q^p)$ is the rectangular partition
with $p$ rows and $q$ columns.

\begin{lemma} \label{L:lincomb}
  Let $R$ be a commutative ring that is complete with respect to the
  ideal $\fm \subset R$, $R = \lim R/\fm^i$, and let $y_1,\dots,y_q
  \in \fm$.  Any symmetric formal power series $f \in R \llbracket
  x_1,\dots,x_p \rrbracket^{\Sigma_p}$ can be written uniquely as an
  (infinite) linear combination
\begin{equation} \label{E:blc} 
  f = \sum_\lambda b_\lambda\, \groth_\lambda(x_1,\dots,x_p \,;\,
  y_1,\dots,y_q) 
\end{equation}
where the sum is over all partitions $\lambda$ with $\ell(\lambda)\leq
p$, and $b_\lambda \in R$.
\end{lemma}
\begin{proof}
  Write $x = (x_1,\dots,x_p)$ and $y = (y_1,\dots,y_q)$.  Set $z =
  (z_1,\dots,z_q)$ where $z_i = 1-(1-y_i)^{-1} = - \sum_{k \geq 1}
  y_i^{\,k} \in R$, which is well defined because $y_i \in \fm$.  If
  $y_1 = \dots = y_q = 0$, then the lemma follows because the term of
  lowest degree in $\groth_\lambda(x)$ is the Schur polynomial
  $s_\lambda(x)$.  Given an expression
\begin{equation} \label{E:bplc}
  f = \sum_\lambda b'_\lambda\, \groth_\lambda(x)
\end{equation}
we can define coefficients $b_\lambda \in R$ by (*) $b_\lambda =
\sum_{\nu,\mu} b'_\nu\, d^\nu_{\lambda \mu}\, \groth_\mu(z)$.  This
infinite sum is well defined in $R$ because $z_i \in \fm$ and
$d^\nu_{\lambda \mu}$ is non-zero only when $|\mu| \geq |\nu| -
|\lambda|$.  By (\ref{E:supersym}) and (\ref{E:split}) we furthermore
have
\[ f = \sum_\nu b'_\nu\, \groth_\nu(z,x \,;\, y)
     = \sum_{\nu,\lambda,\mu} b'_\nu\, d^\nu_{\lambda \mu}\,
       \groth_\mu(z)\, \groth_\lambda(x;y) \\
     = \sum_\lambda b_\lambda\, \groth_\lambda(x ; y) \,.
\]
Similarly, given coefficients $b_\lambda \in R$ such that
(\ref{E:blc}) holds, we obtain coefficients $b'_\lambda \in R$ for
which (\ref{E:bplc}) holds by setting $b'_\lambda = \sum_{\nu,\mu}
b_\nu \, d^\nu_{\lambda \mu} \, \groth_{\mu'}(y)$.  If $f=0$ then all
these coefficients $b'_\lambda$ must be zero.  On the other hand, the
coefficients $b_\lambda$ can be recovered from the $b'_\lambda$ by (*)
since for any fixed partition $\lambda$ we have
\[\begin{split}
\sum_{\nu,\mu} \left(\sum_{\sigma,\tau} b_\sigma\, d^\sigma_{\nu \tau}\,
  \groth_{\tau'}(y) \right) d^\nu_{\lambda \mu}\, \groth_\mu(z)
&= \sum_{\sigma,\nu,\mu,\tau} b_\sigma\, d^\sigma_{\lambda \nu}
  d^\nu_{\mu \tau}\, \groth_\mu(z) \groth_{\tau'}(y) \\
&= \sum_{\sigma,\nu} b_\sigma\, d^\sigma_{\lambda \nu}\, \groth_\nu(z;y)
 = b_\lambda \,.
\end{split}\]
The first equality holds because $\Delta$ is a coproduct and the last
follows from (\ref{E:supersym}) because $\groth_\nu(z;y)$ is equal to
one if $\nu$ is the empty partition and is zero otherwise.
\end{proof}

The stable Grothendieck polynomials given by partitions can be
generalized to stable polynomials $\groth_I$ indexed by arbitrary
finite sequences of integers.  These can be defined by the recursive
identities 
\begin{equation} \label{E:intseq}
  \groth_{I,p,q,J} = \sum_{k=p+1}^q \groth_{I,q,k,J} -
  \sum_{k=p+1}^{q-1} \groth_{I,q-1,k,J}
\end{equation}
whenever $I$ and $J$ are integer sequences and $p < q$ are integers,
as well as the identity $\groth_{I,p} = \groth_I$ for any integer
sequence $I$ and negative integer $p$.  Thus any finite integer
sequence $I$ gives a well defined element $\groth_I \in \Gamma$.  This
notation is required in our formula for quiver coefficients of Dynkin
type given in section \ref{S:formula}.


\section{Quiver coefficients}
\label{S:quivcoef}

In this section we define quiver coefficients and discuss their
conjectured positivity properties.  We start by giving an elementary
construction of the Grothendieck class of an invariant closed
subvariety in a representation.

\subsection{Grothendieck classes}
\label{S:gclass}

Let $G$ be a linear algebraic group over the field $\bk$ and let $V$
be a rational representation of $G$, i.e.\ $V$ is a $\bk$-vector space
of finite dimension and the $G$-action is given by a map of varieties
$G \to \GL(V)$.  Then the coordinate ring $\bk[V] =
\Sym^\bull(V^\vee)$ of polynomial functions on $V$ has a locally
finite linear $G$-action, which in set-theoretic notation is given by
$(g.f)(v) = f(g^{-1}.v)$ for $g\in G$, $f\in \bk[V]$, and $v \in V$.
Locally finite means that $\bk[V]$ is a union of rational
representations of $G$.  Define a {\em $(\bk[V],G)$-module\/} to be a
module $M$ over $\bk[V]$ together with a locally finite linear
$G$-action on $M$ which satisfies that $g.(f\, m) = (g.f)\, (g.m)$ for
$m \in M$.  We will say that $M$ is {\em finitely generated\/} (resp.\ 
{\em free\/}) if this is true as a $\bk[V]$-module.  If $M$ is
finitely generated, then there exists a finite dimensional $G$-stable
vector subspace $U \subset M$ which contains a set of generators.
Notice that $\bk[V] \otimes_\bk U$ has a natural structure of
$(\bk[V],G)$-module, where $\bk[V]$ acts on the first factor and $G$
acts on both factors.  The map $\bk[V]\otimes U \to M$ given by $f
\otimes u \mapsto fu$ is a surjective $G$-equivariant map.  Since $M$
has finite projective dimension as a module over the polynomial ring
$\bk[V]$, and all projective $\bk[V]$-modules are free, it follows
that $M$ has a finite equivariant resolution by finitely generated
free $(\bk[V],G)$-modules.

Let $\Omega \subset V$ be a $G$-stable closed subvariety.  Then the
coordinate ring $\cO_\Omega = \bk[V]/I(\Omega)$ is a finitely
generated $(\bk[V],G)$-module, so it has an equivariant resolution
\begin{equation} \label{E:kvgres}
  0 \to F_p \to F_{p-1} \to \dots \to F_0 \to \cO_\Omega \to 0
\end{equation}
where $F_i$ is a finitely generated free $(\bk[V],G)$-module.  Notice
that $F_i/\fm F_i$ is a rational representation of $G$ for each $i$,
where $\fm = I(0) \subset \bk[V]$ is the maximal ideal of functions
vanishing at the origin of $V$.

Let $\Rep(G)$ be the ring of {\em virtual representations\/} of $G$,
i.e.\ formal linear combinations of irreducible rational
representations.  Multiplication in this ring is defined by tensor
products.  We define the {\em $G$-equivariant Grothendieck class\/} of
$\Omega$ to be the virtual representation
\[ [\cO_\Omega] \ = \ \sum_{i \geq 0} (-1)^i [F_i/\fm F_i] 
   \ \in \Rep(G) \,. 
\]
It follows from results of Thomason \cite{thomason:equivariant*1} that
this class can be identified with the class of the structure sheaf of
$\Omega$ in the equivariant $K$-theory of $V$, see section
\ref{S:degloci}.

\subsection{Classes of quiver cycles}
\label{S:quivclass}

Let $V = \bigoplus_{a \in Q_1} \Hom(E_{t(a)}, E_{h(a)})$ be the vector
space of representations of the quiver $Q$.  Then $V$ is a rational
representation of the group $\bG = \prod_{i=1}^n \GL(E_i)$.  It
follows that any quiver cycle $\Omega \subset V$ defines a
Grothendieck class $[\cO_\Omega] \in \Rep(\bG)$.

Choose a decomposition of each vertex vector space as a sum of one
dimensional vector spaces, $E_i = L^i_1 \oplus \dots \oplus
L^i_{e_i}$, and let $\bT \subset \bG$ be the maximal torus that
preserves these decompositions.  Then the virtual representations of
$\bT$ form the Laurent polynomial ring $\Rep(\bT) =
\Z[\hmm{.5}[L^i_j]^{\pm 1}]$.  It follows from
\cite[Cor.~II.2.7]{jantzen:representations*1} that the restriction
map $\Rep(\bG) \to \Rep(\bT)$ is injective, and the image
must consist of Laurent polynomials that are simultaneously symmetric
in each group of variables $\{[L^i_1], \dots, [L^i_{e_i}]\}$.  Since
all such polynomials can be generated by the exterior powers
$[\bigwedge^j E_i] \in \Rep(\bG)$, it follows that $\Rep(\bG) \subset
\Rep(\bT)$ is the subring of simultaneously symmetric Laurent
polynomials.

Set $x^i_j = 1 - [L^i_j]^{-1}$ for $1 \leq i \leq n$ and $1 \leq j
\leq e_i$, and let $\Z\llbracket x^i_j \rrbracket$ be the ring of
formal power series in these variables.  We will consider $\Rep(\bT)$
as a subring of $\Z\llbracket x^i_j \rrbracket$, with $[L^i_j] =
\sum_{p \geq 0} (x^i_j)^p$.  In particular, the Grothendieck class
$[\cO_\Omega]$ can be regarded as a power series in $\Z\llbracket
x^i_j \rrbracket$.  The $\bT$-equivariant cohomology of $V$ can be
identified with the polynomial ring $H^*_\bT(V) = \Z[x^i_j]$, and
$H^*_\bG(V) \subset H^*_\bT(V)$ is the subring of simultaneously
symmetric polynomials.  The power series $[\cO_\Omega] \in
\Z\llbracket x^i_j \rrbracket$ has no non-zero terms of total degree
smaller than $d = \codim(\Omega;V)$, and the term of degree $d$ is the
cohomology class $[\Omega] \in H^d_{\bG}(V)$, see section
\ref{S:interpret}.

If $U$ is any rational representation of $\bG$, we can write it as a
direct sum of one dimensional $\bT$-representations, $U = L_1 \oplus
\dots \oplus L_u$.  Given a partition $\nu$ we then define
$\groth_\nu(U) = \groth_\nu(1-[L_1]^{-1}, \dots, 1-[L_u]^{-1})
\in \Rep(\bG) \subset \Rep(\bT)$.  For example, $\groth_\nu(E_i) =
\groth_\nu(x^i_1, \dots, x^i_{e_i})$.  More generally, given two
rational $\bG$-representations $U_1$ and $U_2$ we define
\begin{equation} \label{E:groth_virt}
  \groth_\nu(U_1-U_2) \ = \ \sum_{\lambda,\mu} d^\nu_{\lambda \mu} \,
  \groth_\lambda(U_1) \, \groth_{\mu'}(U_2^\vee) \ \in \Rep(\bG)
\end{equation}
where $U_2^\vee$ is the dual representation of $U_2$.  The Schur
function $s_\nu(U_1 - U_2)$ is defined as the term of total (and
lowest) degree $|\nu|$ in $\groth_\nu(U_1 - U_2)$ when considered as a
power series in $\Z\llbracket x^i_j \rrbracket$.

From now on we assume that $Q$ is a quiver without oriented loops.
Our definition of quiver coefficients is based on the following
proposition.  Recall that we set $M_i = \bigoplus_{a : h(a)=i}
E_{t(a)}$ for $i \in Q_0$.

\begin{prop} \label{P:lincomb}
  Let $Q$ be a quiver without oriented loops.  Every element of
  $\Rep(\bG)$ can be expressed uniquely as a (possibly infinite)
  $\Z$-linear combination of products
\[ \groth_{\mu_1}(E_1-M_1) \, \groth_{\mu_2}(E_2-M_2) \cdots
   \groth_{\mu_n}(E_n-M_n)
\]
given by partitions $\mu_1,\dots,\mu_n$ such that $\ell(\mu_i) \leq
e_i$ for each $i$.
\end{prop}
\begin{proof}
  Let $l \in Q_0$ be a vertex which is not the tail of any arrow in
  $Q$.  Since every element of $\Rep(\bG) \subset \Z\llbracket x^i_j
  \rrbracket $ is symmetric in the variables $x^l_1, \dots,
  x^l_{e_l}$, we can use Lemma~\ref{L:lincomb} to write it as an
  (infinite) linear combination of the elements $\groth_{\mu_l}(E_l -
  M_l)$ given by partitions $\mu_l$ with at most $e_l$ rows, and with
  coefficients in the subring $R = \Z\llbracket x^i_j : i \neq l
  \rrbracket$.  By induction on $n$, applied to the quiver obtained
  from $Q$ by removing the vertex $l$ and all arrows to it, it follows
  that each of the coefficients are unique $\Z$-linear combinations of
  the products $\prod_{i \neq l} \groth_{\mu_i}(E_i-M_i)$.
\end{proof}

\begin{defn} \label{D:quivcoef}
  Let $\Omega \subset V$ be a quiver cycle for a quiver $Q$ without
  oriented loops.  The {\em quiver coefficients\/} of
  $\Omega$ are the unique integers $c_\mu(\Omega) \in \Z$, indexed by
  sequences $\mu = (\mu_1,\dots,\mu_n)$ of partitions $\mu_i$ with
  $\ell(\mu_i) \leq e_i$, such that
  \[ [\cO_\Omega] \ = \ \sum_\mu c_\mu(\Omega) \,
  \groth_{\mu_1}(E_1-M_1)\, \groth_{\mu_2}(E_2-M_2) \cdots
  \groth_{\mu_n}(E_n-M_n) \ \in \Rep(\bG) \,.
  \]
  The {\em cohomological quiver coefficients\/} of $\Omega$ are the
  coefficients $c_\mu(\Omega)$ for which $\sum |\mu_i| =
  \codim(\Omega)$.
\end{defn}

It follows from Corollary \ref{C:univlocus} below that these
coefficients generalize the equioriented quiver coefficients from
\cite{buch.fulton:chern, buch:grothendieck}.  The cohomological quiver
coefficients determine the cohomology class of $\Omega$ as
\[ [\Omega] = \sum_{\sum |\mu_i|=\codim(\Omega)} c_\mu(\Omega)\,
   s_{\mu_1}(E_1-M_1) \, s_{\mu_2}(E_2-M_2) \cdots s_{\mu_n}(E_n-M_n)
   \in H^*_\bG(V) \,.
\]

\begin{example} \label{E:porteous}
  Let $Q = \{ 1 \to 2 \}$ be a quiver of type A$_2$.  Then any quiver
  cycle in $V = \Hom(E_1,E_2)$ has the form $\Omega_r = \{ \phi \in V
  \mid \rank(\phi) \leq r \}$.  It follows from the Thom-Porteous
  formula of \cite[Thm.~2.3]{buch:grothendieck} and
  Corollary~\ref{C:univlocus} that $[\cO_{\Omega_r}] = \groth_{R}(E_2
  - E_1)$, where $R = (e_1-r)^{e_2-r}$ is the rectangular partition
  with $e_2-r$ rows and $e_1-r$ columns.  We have $c_{(R)}(\Omega_r) =
  1$, and all other quiver coefficients of $\Omega_r$ are zero.
\end{example}

\subsection{Properties of quiver coefficients}

We do not know a good reason why the quiver coefficients should
satisfy the finiteness and positivity properties stated in
Conjecture~\ref{C:altsign}.  In the case of equioriented quivers where
this conjecture is known, these properties are consequences of
explicit formulas for quiver coefficients that are proved with a
combination of geometric and combinatorial methods.  This is also true
for our proof of the finiteness part (a) for quivers of Dynkin type in
section \ref{S:formula}, and our proof of the full conjecture for
quivers of type A$_3$ in section \ref{S:a3}.  However, if the full
conjecture is true, then it is natural to expect that some underlying
geometric principle is in play.

One might try to express the classes of quiver cycles as linear
combinations of other products of Grothendieck polynomials than those
used in Definition~\ref{D:quivcoef}, but most choices do not lead to
finiteness or positivity properties of the coefficients (or they lead
to such properties that follow from Conjecture~\ref{C:altsign}).  The
one interesting alternative choice that we know about is to define
{\em dual quiver coefficients\/} $\wt c_\mu(\Omega)$ of a quiver cycle
$\Omega$ by the identity
\[ [\cO_\Omega] = \sum_\mu \wt c_\mu(\Omega) \,
   \groth_{\mu_1}(N_1 - E_1)\, \groth_{\mu_2}(N_2 - E_2) \cdots 
   \groth_{\mu_n}(N_n - E_n) \,,
\]
where the sum is over sequences $\mu = (\mu_1,\dots,\mu_n)$ of
partitions such that $\mu_i$ has at most $e_i$ columns for each $i$,
and $N_i = \bigoplus_{a:t(a)=i} E_{h(a)}$.  These dual coefficients
are nothing but the ordinary quiver coefficients for $\Omega$ when
considered as a cycle of quiver representations on the dual vector
spaces $E_i^\vee$, for the quiver obtained from $Q$ by reversing all
arrows.  This follows from the identity $\groth_\lambda(U_1-U_2) =
\groth_{\lambda'}(U_2^\vee-U_1^\vee)$ which holds for arbitrary
rational representations $U_1$ and $U_2$ of $\bG$ \cite[Lemma
3.4]{buch:littlewood-richardson}.  We note that for an equioriented
quiver $Q = \{ 1 \to 2 \to \dots \to n\}$, the two notions of quiver
coefficients also coincide without modifying the quiver.  In fact, an
equioriented coefficient $c_{(\mu_1,\dots,\mu_n)}(\Omega)$ is non-zero
only if $\mu_1$ is the empty partition, in which case we have
$c_{(\emptyset,\mu_2,\dots,\mu_n)}(\Omega) = \wt
c_{(\mu_2,\dots,\mu_n,\emptyset)}(\Omega)$.  On the other hand, for
quivers that are not equioriented, it appears to be difficult to
relate the properties of quiver coefficients and dual quiver
coefficients of the same quiver cycle.  For the simplest example, the
reader is invited to compare the formulas for inbound and outbound
A$_3$-quivers proved in section \ref{S:a3}.

It is convenient to encode the quiver coefficients for $\Omega$ as a
linear combination of tensors,
\begin{equation}
   P_\Omega = \sum_\mu c_\mu(\Omega)\, \groth_{\mu_1}\otimes
  \groth_{\mu_2} \otimes \dots \otimes \groth_{\mu_n} \,.
\end{equation}
If Conjecture~\ref{C:altsign} (a) is true, then this is an element of
the tensor power $\Gamma^{\otimes n}$ of the ring of stable
Grothendieck polynomials $\Gamma$; otherwise $P_\Omega$ lives in a
completion of this ring.  We will use the notation that for any linear
combination $P = \sum_\mu c_\mu\, \groth_{\mu_1} \otimes \dots \otimes
\groth_{\mu_n}$ and classes $\alpha_1,\dots,\alpha_n \in \Rep(\bG)$, we
set $P(\alpha_1, \dots, \alpha_n) = \sum_\mu c_\mu\,
\groth_{\mu_1}(\alpha_1) \cdots \groth_{\mu_n}(\alpha_n)$.  The
definition of quiver coefficients then states that $[\cO_\Omega] =
P_\Omega(E_1-M_1,\dots,E_n-M_n) \in \Rep(\bG)$.

In addition to the evidence for Conjecture~\ref{C:altsign} mentioned
above, we have used Macaulay 2 \cite{grayson.stillman:macaulay} and
other software to compute the quiver coefficients of many quiver
cycles, including some that are not orbit closures (and not of Dynkin
type).  In almost all cases where Macaulay 2 was able to produce a
free resolution of the coordinate ring of a quiver cycle, we could
convert the corresponding expression for its Grothendieck class into a
{\em finite\/} linear combination of products of Grothendieck
polynomials as in Definition~\ref{D:quivcoef}.  In a few cases we did
not succeed in this, but expect that this was caused by lack of
computing power.  We have never encountered any negative cohomological
quiver coefficients; and when the general quiver coefficients failed
to have alternating signs, we could often show that the corresponding
quiver cycle did not have rational singularities, for example by using
Brion's theorem described in the introduction \cite{brion:positivity}.

\begin{example}
Let $Q = \{ 1 \toto 2 \}$ be the Kronecker quiver and fix the
dimension vector $e = (3,3)$.  Let $\Omega \subset V$ be the closure
of the orbit through the point
\[ \left( 
   \left[\begin{smallmatrix} 0 & 0 & 0 \\ 1 & 0 & 0 \\ 0 & 1 & 0
   \end{smallmatrix}\right] ,
   \left[\begin{smallmatrix} 1 & 0 & 0 \\ 0 & 0 & 0 \\ 0 & 0 & 1 
   \end{smallmatrix}\right]
   \right) \,.
\]
Zwara has shown in \cite{zwara:orbit} that this orbit closure has ugly
singularities, in particular they are not rational.  With help from
Macaulay 2 \cite{grayson.stillman:macaulay}, we have determined the
quiver coefficients for $\Omega$.  There are finitely many of them,
and they are encoded in the following expression $P_\Omega$ satisfying
that $P_\Omega(E_1\,, E_2-E_1\oplus E_1) = [\cO_\Omega]$\,:
\vmm{1}

\newcommand{\dgbrk}{\\ \hmm{10}}
\newcommand{\lnbrk}{\\ \hmm{13}}
\noindent
$\hmm{2} P_\Omega = \,
3 \otimes \groth_{3,1}
+4\,\groth_{1} \otimes \groth_{3}
+1 \otimes \groth_{2,2}
+2\,\groth_{1} \otimes \groth_{2,1} 
+3\,\groth_{2} \otimes \groth_{2}
+\groth_{2} \otimes \groth_{1,1} \lnbrk  
+2\,\groth_{3} \otimes \groth_{1}
+\groth_{4} \otimes 1
\dgbrk
-3 \otimes \groth_{3,2}
-8\,\groth_{1} \otimes \groth_{3,1}
-6\,\groth_{2} \otimes \groth_{3}
-2\,\groth_{1} \otimes \groth_{2,2}
-5\,\groth_{2} \otimes \groth_{2,1}
-4\,\groth_{3} \otimes \groth_{2} \lnbrk  
-2\,\groth_{3} \otimes \groth_{1,1}
-2\,\groth_{4} \otimes \groth_{1}
\dgbrk
-1 \otimes \groth_{4,2}
-3 \otimes \groth_{4,1,1}
-6\,\groth_{1} \otimes \groth_{4,1}
-3\,\groth_{2} \otimes \groth_{4}
-6\,\groth_{1,1} \otimes \groth_{4} 
+4\,\groth_{1} \otimes \groth_{3,2} \lnbrk  
+7\,\groth_{2} \otimes \groth_{3,1}
+2\,\groth_{3} \otimes \groth_{3}
+\groth_{2} \otimes \groth_{2,2}
+4\,\groth_{3} \otimes \groth_{2,1}
+\groth_{4} \otimes \groth_{2}
+\groth_{4} \otimes \groth_{1,1}
\dgbrk
+1 \otimes \groth_{4,3}
+5 \otimes \groth_{4,2,1}
+10\,\groth_{1} \otimes \groth_{4,2}
+10\,\groth_{1} \otimes \groth_{4,1,1}
+14\,\groth_{2} \otimes \groth_{4,1} \lnbrk  
+15\,\groth_{1,1} \otimes \groth_{4,1} 
+4\,\groth_{3} \otimes \groth_{4}
+12\,\groth_{2,1} \otimes \groth_{4}
-\groth_{2} \otimes \groth_{3,2} \lnbrk  
-2\,\groth_{3} \otimes \groth_{3,1}
-\groth_{4} \otimes \groth_{2,1}
\dgbrk
-2 \otimes \groth_{4,3,1}
-4\,\groth_{1} \otimes \groth_{4,3}
-1 \otimes \groth_{4,2,2}
-16\,\groth_{1} \otimes \groth_{4,2,1}
-16\,\groth_{2} \otimes \groth_{4,2} \lnbrk  
-12\,\groth_{1,1} \otimes \groth_{4,2} 
-12\,\groth_{2} \otimes \groth_{4,1,1}
-10\,\groth_{1,1} \otimes \groth_{4,1,1}
-10\,\groth_{3} \otimes \groth_{4,1} \lnbrk  
-29\,\groth_{2,1} \otimes \groth_{4,1} 
-\groth_{4} \otimes \groth_{4} 
-7\,\groth_{3,1} \otimes \groth_{4} 
-3\,\groth_{2,2} \otimes \groth_{4}
\dgbrk
+1 \otimes \groth_{4,3,2}
+6\,\groth_{1} \otimes \groth_{4,3,1}
+5\,\groth_{2} \otimes \groth_{4,3}
+3\,\groth_{1,1} \otimes \groth_{4,3}
+2\,\groth_{1} \otimes \groth_{4,2,2} \lnbrk  
+18\,\groth_{2} \otimes \groth_{4,2,1} 
+14\,\groth_{1,1} \otimes \groth_{4,2,1}
+8\,\groth_{3} \otimes \groth_{4,2}
+22\,\groth_{2,1} \otimes \groth_{4,2} \lnbrk  
+6\,\groth_{3} \otimes \groth_{4,1,1} 
+18\,\groth_{2,1} \otimes \groth_{4,1,1} 
+2\,\groth_{4} \otimes \groth_{4,1} 
+16\,\groth_{3,1} \otimes \groth_{4,1} \lnbrk  
+6\,\groth_{2,2} \otimes \groth_{4,1}
+\groth_{4,1} \otimes \groth_{4}
+3\,\groth_{3,2} \otimes \groth_{4}
\dgbrk
-2\,\groth_{1} \otimes \groth_{4,3,2}
-6\,\groth_{2} \otimes \groth_{4,3,1}
-4\,\groth_{1,1} \otimes \groth_{4,3,1}
-2\,\groth_{3} \otimes \groth_{4,3}
-5\,\groth_{2,1} \otimes \groth_{4,3} \lnbrk  
-\groth_{2} \otimes \groth_{4,2,2} 
-\groth_{1,1} \otimes \groth_{4,2,2}
-8\,\groth_{3} \otimes \groth_{4,2,1}
-24\,\groth_{2,1} \otimes \groth_{4,2,1}
-\groth_{4} \otimes \groth_{4,2}  \lnbrk  
-11\,\groth_{3,1} \otimes \groth_{4,2}
-3\,\groth_{2,2} \otimes \groth_{4,2} 
-\groth_{4} \otimes \groth_{4,1,1}
-9\,\groth_{3,1} \otimes \groth_{4,1,1} \lnbrk  
-3\,\groth_{2,2} \otimes \groth_{4,1,1}
-2\,\groth_{4,1} \otimes \groth_{4,1}
-6\,\groth_{3,2} \otimes \groth_{4,1}
\dgbrk
+\groth_{2} \otimes \groth_{4,3,2}
+\groth_{1,1} \otimes \groth_{4,3,2}
+2\,\groth_{3} \otimes \groth_{4,3,1}
+6\,\groth_{2,1} \otimes \groth_{4,3,1}
+2\,\groth_{3,1} \otimes \groth_{4,3} \lnbrk  
+\groth_{2,1} \otimes \groth_{4,2,2}
+\groth_{4} \otimes \groth_{4,2,1}
+11\,\groth_{3,1} \otimes \groth_{4,2,1}
+3\,\groth_{2,2} \otimes \groth_{4,2,1} \lnbrk  
+\groth_{4,1} \otimes \groth_{4,2}
+3\,\groth_{3,2} \otimes \groth_{4,2}
+\groth_{4,1} \otimes \groth_{4,1,1}
+3\,\groth_{3,2} \otimes \groth_{4,1,1}
\dgbrk
-\groth_{2,1} \otimes \groth_{4,3,2}
-2\,\groth_{3,1} \otimes \groth_{4,3,1}
-\groth_{4,1} \otimes \groth_{4,2,1}
-3\,\groth_{3,2} \otimes \groth_{4,2,1}
$
\vmm{1}

We note that while this expression fails to have alternating signs,
the signs are still periodic in a curious way.  In fact, the terms
$\groth_\lambda \otimes \groth_\nu$ of $P_\Omega$ displayed above are
sorted according to the lexicographic order on the partitions, with
$\nu$ taking precedence over $\lambda$, which makes the periodicity
readily visible.  Furthermore, starting from the degree 8 term, the
signs of the quiver coefficients are the opposite of the expected.  We
have also observed this phenomenon for other quiver cycles without
rational singularities, but have no explanation for it.

Our calculation also shows that $\Omega$ is the cone over a subvariety
of $\bP^{17}$ with Grothendieck class equal to 
\[ 51\, h^4 - 132\, h^5 + 70\, h^6 + 144\, h^7 - 261\, h^8 + 184\, h^9
   - 66\, h^{10} + 12\, h^{11} - h^{12}
\]
where $h$ is the class of a hyperplane.  Using Brion's result
\cite{brion:positivity}, this gives an alternative proof that $\Omega$
lacks rational singularities.

Finally, if the cohomology class of $\Omega$ is expressed in the basis
of products $s_{\mu_1}(E_1)\, s_{\mu_2}(E_2-E_1)$, then we obtain
\[\begin{split}
  [\Omega] = & \ 
  3 s_{3, 1}(E_2\!-\!E_1)
  + s_{1}(E_1) s_{3}(E_2\!-\!E_1)
  + s_{2, 2}(E_2\!-\!E_1)
  - 2 s_{1}(E_1) s_{2, 1}(E_2\!-\!E_1) \\ &
  - 2 s_{1, 1}(E_1) s_{2}(E_2\!-\!E_1)
  + s_{1, 1}(E_1) s_{1, 1}(E_2\!-\!E_1)
  + 3 s_{1, 1, 1}(E_1) s_{1}(E_2\!-\!E_1) \,.
\end{split}\] 
This illustrates that our choice of basis is essential to the
positivity conjecture.  It is also essential to the finiteness
conjecture, since in general it requires an infinite linear
combination of products $\groth_{\mu_1}(E_1) \groth_{\mu_2}(E_2-E_1)$
to express a class $\groth_\lambda(E_2-E_1\oplus E_1)$.
\end{example}


\section{Degeneracy loci}
\label{S:degloci}

This section interprets quiver coefficients as formulas for degeneracy
loci defined by quivers of vector bundles over a base variety.  We
start by summarizing some facts about equivariant $K$-theory of
schemes based on Thomason's paper \cite{thomason:equivariant*1}.

\subsection{$K$-theory} \label{S:ktheory}

Let $G$ be an algebraic group over the field $\bk$ and let $X$ be an
algebraic $G$-scheme over $\bk$.  A $G$-{\em equivariant sheaf\/} on
$X$ is a coherent $\cO_X$-module $\cF$ together with a given
isomorphism $I : a^* \cF \cong p_2^* \cF$, where $a : G \times X \to
X$ is the action and $p_2 : G \times X \to X$ is the projection.  This
isomorphism must satisfy that $(m \times \id_X)^* I = p_{23}^* I \circ
(\id_G \times a)^* I$ as morphisms of sheaves on $G \times G \times
X$, where $m$ is the group operation on $G$ and $p_{23}$ is the
projection to the last two factors of $G \times G \times X$.  A
$G$-{\em equivariant vector bundle\/} on $X$ is a locally free
$G$-equivariant sheaf of constant rank.

The {\em $G$-equivariant $K$-homology\/} of $X$ is the Grothendieck
group $K_G(X)$ generated by isomorphism classes of $G$-equivariant
sheaves, modulo relations saying that $[\cF] = [\cF'] + [\cF'']$ if
there exists a $G$-equivariant short exact sequence $0 \to \cF' \to
\cF \to \cF'' \to 0$.  The {\em $G$-equivariant $K$-cohomology\/} of
$X$ is the Grothendieck ring $K^G(X)$ of $G$-equivariant vector
bundles.  The group $K_G(X)$ is a module over the ring $K^G(X)$; both
the ring structure of $K^G(X)$ and its action on $K_G(X)$ are defined
by tensor products.  If $X$ is a non-singular variety and $G$ is a
linear algebraic group, then the implicit map $K^G(X) \to K_G(X)$ that
sends an equivariant vector bundle to its sheaf of sections is an
isomorphism \cite[Thm.~1.8]{thomason:equivariant*1}.  The equivariant
$K$-theory of a point is the ring $K^G(\pt) = \Rep(G)$ of virtual
representations of $G$.  Any $G$-equivariant map $f : X \to Y$ defines
a ring homomorphism $f^* : K^G(Y) \to K^G(X)$ given by pullback of
vector bundles.  If $f$ is flat then it also defines a pullback map
$f^* : K_G(Y) \to K_G(X)$ on Grothendieck groups.  The same is true if
$f$ is a regular embedding, in which case the pullback is given by
$f^*[\cF] = \sum_{i\geq 0} (-1)^i [\Tor^Y_i(\cO_X,\cF)]$.  A proper
equivariant map $f : X \to Y$ defines a pushforward map $f_* : K_G(X)
\to K_G(Y)$ given by $f_*[\cF] = \sum_{i \geq 0} (-1)^i [R^i f_*
\cF]$.  This pushforward map is a homomorphism of $K^G(Y)$-modules by
the projection formula.  If $\pi : E \to X$ is (the total space of) a
$G$-equivariant vector bundle, then $\pi^* : K_G(X) \to K_G(E)$ is an
isomorphism \cite[Thm.~1.7]{thomason:equivariant*1}, and we will
identify $K_G(E)$ with $K_G(X)$ using this map.  The inverse map is
pullback along any equivariant section $X \to E$.  When $G = \{e\}$ is
the trivial group, we will use the notation $K^\circ(X) =
K^{\{e\}}(X)$ and $K_\circ(X) = K_{\{e\}}(X)$ for the ordinary
$K$-theory groups of $X$.

Stable Grothendieck polynomials can be used to define $K$-theory
classes as follows.  Given a vector bundle over $X$ which can be
written as a direct sum of line bundles $\cE = \calL_1 \oplus \dots
\oplus \calL_r$, and a partition $\nu$, we define
\begin{equation} \label{E:grothbdl}
  \groth_\nu(\cE) \ = \ \groth_\nu(1-\calL_1^{-1}, \dots, 1-\calL_r^{-1})
  \ \in K^\circ(X) \,.
\end{equation}
The symmetry of $\groth_\nu$ implies that this class is a polynomial
in the exterior powers of the dual bundle $\cE^\vee$, so it is well
defined even when $\cE$ is not a direct sum of line bundles.
Furthermore, if $X$ is a $G$-scheme and $\cE$ is a $G$-equivariant
vector bundle, then (\ref{E:grothbdl}) defines a class
$\groth_\nu(\cE) \in K^G(X)$.  Given two $G$-vector bundles $\cE_1$
and $\cE_2$ we define
\begin{equation} \label{E:groth_double}
  \groth_\nu(\cE_1 - \cE_2) \ = \ \sum_{\lambda,\mu} d^\nu_{\lambda \mu}
  \, \groth_\lambda(\cE_1) \, \groth_{\mu'}(\cE_2^\vee) \ \in K^G(X) \,.
\end{equation}
This extends (\ref{E:groth_virt}).  The linear map $\Gamma \to K^G(X)$
given by $\groth_\nu \mapsto \groth_\nu(\cE_1 - \cE_2)$ is a ring
homomorphism.  The identity (\ref{E:supersym}) implies that
$\groth_\nu(\cE_1\oplus \cE_3 - \cE_2\oplus \cE_3) =
\groth_\nu(\cE_1-\cE_2)$ for any third $G$-vector bundle $\cE_3$.
Equivalently, the stable Grothendieck polynomial for $\nu$ defines a
linear operator $\groth_\nu : K^G(X) \to K^G(X)$.  Equation
(\ref{E:split}) implies that $\groth_\nu(\alpha + \beta) =
\sum_{\lambda,\mu} d^\nu_{\lambda \mu} \groth_\lambda(\alpha)
\groth_\mu(\beta)$ for all classes $\alpha,\beta \in K^G(X)$.

\subsection{Interpretations of Grothendieck classes}
\label{S:interpret}

Assume that $G$ is a connected reductive linear algebraic group
containing a $\bk$-split maximal torus $T \subset G$, i.e.\ $T \cong
(\bG_m)^r$ is defined over $\bk$.  Let $V$ be a rational
representation of $G$ and let $\Omega \subset V$ be a $G$-stable
closed subvariety.  Then the structure sheaf $\cO_\Omega$ is a
$G$-equivariant sheaf on $V$, so it defines a class $[\cO_\Omega] \in
K_G(V)$.  If we use the fact that $V$ is an equivariant vector bundle
over a point to identify $K_G(V)$ with $\Rep(G)$, then this class
agrees with the Grothendieck class of $\Omega$ defined in Section
\ref{S:gclass}.

Let $X$ be an algebraic scheme equipped with a principal $G$-bundle $P
\to X$, i.e.\ $G$ acts freely on $P$ and $X$ equals $P/G$ as a
geometric quotient \cite{mumford:geometric}.  For a $G$-variety $Y$ we
write $Y_G = P \times^G Y = (P\times Y)/G$.  We will use this notation
only when $Y$ is equivariantly embedded as a closed subvariety of a
non-singular variety, in which case it follows from
\cite[Prop.~23]{edidin.graham:equivariant} that $Y_G$ is defined as a
scheme.  Using that the category of $G$-equivariant sheaves on $P$ is
equivalent to the category of coherent $\cO_X$-modules
\cite[Thm.~6.1.4]{bosch.lutkebohmert.ea:neron}, it follows that $V_G$
is a vector bundle over $X$ with fibers isomorphic to $V$
\cite[Lemma~1]{edidin.graham:equivariant}, and the closed subscheme
$\Omega_G \subset V_G$ is a {\em translated degeneracy locus},
consisting of one copy of $\Omega$ in each fiber.  It's structure
sheaf defines a Grothendieck class $[\cO_{\Omega_G}] \in K_\circ(V_G)
= K_\circ(X)$.

More generally, let $H$ be a second algebraic group over $\bk$, and
assume that $P$ and $X$ are $H$-schemes so that the map $P \to X$ is
equivariant and the $H$-action on $P$ commutes with the $G$-action.
In this case $V_G$ is an $H$-vector bundle over $X$, and $\Omega_G$
defines an equivariant class $[\cO_{\Omega_G}] \in K_H(V_G) = K_H(X)$.
Let $\phi_G : \Rep(G) \to K^H(X)$ be the ring homomorphism defined by
$\phi_G(U) = [U_G]$ for any rational $G$-representation $U$.  The
following lemma interprets the Grothendieck class $[\cO_\Omega] \in
\Rep(G)$ as a formula for degeneracy loci.

\begin{prop} \label{P:gclass_equiv}
  The $H$-equivariant Grothendieck class of $\Omega_G \subset V_G$ is
  given by $[\cO_{\Omega_G}] = \varphi_G([\cO_\Omega]) \in K_H(X)$.
\end{prop}
\begin{proof}
  A finitely generated free $(\bk[V],G)$-module $F$ corresponds to a
  $G$-equivariant vector bundle $\wt F = \Spec(\Sym^\bull F^\vee)$
  over $V$, which in turn defines the $H$-equivariant vector bundle
  $\wt F_G = P \times^G \wt F$ on $V_G$
  \cite[Lemma~1]{edidin.graham:equivariant}.  This construction
  applied to (\ref{E:kvgres}) produces an exact sequence
  \[ 0 \to (\wt F_r)_G \to (\wt F_{r-1})_G \to \dots \to 
     (\wt F_0)_G \to \cO_{\Omega_G} \to 0
  \]
  of $H$-equivariant coherent sheaves on $V_G$.  Let $s : X \to V_G$
  be the zero section.  Since the fiber of $\wt F_i$ over the origin
  of $V$ equals $F_i/\fm F_i$, it follows that $s^*(\wt F_i)_G =
  (F_i/\fm F_i)_G$.  We conclude that
  \[ [\cO_{\Omega_G}] = \sum_{i \geq 0} (-1)^i\, s^*[(\wt F_i)_G] =
    \sum_{i \geq 0} (-1)^i [(F_i/\fm F_i)_G] = \varphi_G([\cO_\Omega])
  \]
  in $K_H(X)$, as required.
\end{proof}

Write $T = (\bG_m)^r$ as a product of multiplicative groups, and
define one-dimen\-sional $T$-representations $L_1, \dots, L_r$ by $L_i
= \bk$ and $(t_1,\dots,t_r).v = t_i v$ for $v \in L_i$.  Then we have
$\Rep(T) = \Z[L_1^{\pm 1}, \dots, L_r^{\pm 1}] \subset \Z\llbracket
x_1,\dots,x_r \rrbracket$ where $x_i = 1 - L_i^{-1}$.  Since $\Rep(G)
\subset \Rep(T)$ by \cite[Cor.~II.2.7]{jantzen:representations*1}, we
may regard the class $[\cO_\Omega]$ as a power series.

The variety $\Omega \subset V$ also defines a class $[\Omega]$ in the
equivariant Chow cohomology ring $H_T^*(V)$.  If we abuse notation and
write $x_i$ also for the Chern root $c_1(L_i) \in H_T^*(\pt) =
H_T^*(V)$, then this ring is the polynomial ring $H_T^*(V) =
\Z[x_1,\dots,x_r]$ by \cite[\S15]{totaro:chow}, and the class
$[\Omega]$ coincides with the term of total degree $d =
\codim(\Omega;V)$ in the power series $[\cO_\Omega]$.  To see this, we
need Totaro's algebraic approximation of the classifying space for $T$
\cite{totaro:chow}.  Set $P = \prod_{i=1}^r \left( L_i^{\oplus d+1}
  \smallsetminus \{0\} \right)$ and $X = P/T = \prod_{i=1}^r \bP^d$.
Then $H_T^i(V) = H^i(V_T) = H^i(X)$ for $i \leq d$ by
\cite[Thm.~1.1]{totaro:chow} or
\cite[Prop.~4]{edidin.graham:equivariant}, where $V_T = P \times^T V$
and $x_i \in H_T^i(V)$ corresponds to a hyperplane class in the $i$-th
factor of $X$.  The cohomology class of $\Omega$ is defined by
$[\Omega] := [\Omega_T] \in H^d(V_T)$.  Let $\ch : K^\circ(V_T) \to
H^*(V_T) \otimes \Q$ be the Chern character, i.e.\ the ring
homomorphism defined formally by $\ch(\calL) = \exp(c_1(\calL))$ for
any line bundle $\calL$ on $V_T$ \cite[Ex.\ 
3.2.3]{fulton:intersection}.  Then we have $\ch(\varphi_T(x_i)) = 1 -
\exp(-x_i)$, so the lowest term of $[\cO_\Omega]$ agrees with the
lowest term of $\ch(\varphi_T([\cO_\Omega]))$.  Now
Proposition~\ref{P:gclass_equiv} and
\cite[Ex.~15.2.16]{fulton:intersection} imply that
$\ch(\varphi_T([\cO_\Omega])) = \ch([\cO_{\Omega_T}]) = [\Omega_T]$ +
higher terms.  This shows that $[\Omega]$ is the lowest term in
$[\cO_\Omega]$, and also that $[\cO_\Omega]$ has no non-zero terms of
degree smaller than $\codim(\Omega;V)$.

We finally prove that the Grothendieck class of $\Omega$ is uniquely
determined by the formula it provides in ordinary $K$-theory.

\begin{prop} \label{P:gclass_unique}
  The equivariant Grothendieck class of $\Omega$ is the unique virtual
  representation $[\cO_\Omega] \in \Rep(G)$ for which
  $[\cO_{\Omega_G}] = \varphi_G([\cO_\Omega]) \in K_\circ(X)$ for
  every non-singular variety $X$ and principal $G$-bundle $P \to X$.
\end{prop}
\begin{proof}
  In view of Proposition~\ref{P:gclass_equiv}, it is enough to show
  that if $\alpha \neq 0 \in \Rep(G)$, then for some principal
  $G$-bundle $P \to X$ with $X$ non-singular we have
  $\varphi_G(\alpha) \neq 0 \in K_\circ(X)$.
  
  Let $d$ be the degree of the lowest non-zero term of $\alpha \in
  \Z\llbracket x_1,\dots,x_r \rrbracket$.  As in
  \cite[Lemma~9]{edidin.graham:equivariant} we embed $G$ in $\GL(m)$
  for some $m$ and let $P$ be the set of all $m \times (m+d)$ matrices
  of full rank.  Then $G$ acts freely on $P$, the quotients $X = P/G$
  and $P/T$ are non-singular varieties, and since $P$ has codimension
  $d+1$ in the vector space of all $m \times (m+d)$ matrices, it
  follows from \cite[Thm.~1.1]{totaro:chow} or
  \cite[Prop.~4]{edidin.graham:equivariant} that $H^i(P/T) = H^i_T(V)$
  for $i \leq d$.  Consider the commutative diagram
  \[ \xymatrix{
    \Rep(G) \ar[r] \ar[d]^{\varphi_G} & \Rep(T) \ar[d]^{\varphi_T} \\
    K^\circ(X) \ar[r] & K^\circ(P/T) \ar[r]^{\ch\ \ \ } & 
    H^*(P/T) \otimes \Q
    } 
  \]
  where the bottom-left map is pullback along $P/T \to P/G = X$.
  Since the image of $\alpha$ in $H^d(P/T)\otimes \Q = H^d_T(V)
  \otimes \Q$ is non-zero, we conclude that $\varphi_G(\alpha) \in
  K^\circ(X) = K_\circ(X)$ is non-zero as well.
\end{proof}

\subsection{Degeneracy loci defined by quiver cycles}

Let $V$ and $\bG$ be as in section \ref{S:quivclass}, and let $\Omega
\subset V$ be a quiver cycle.  We will use the constructions given
above to interpret the quiver coefficients of $\Omega$ in terms of
formulas for degeneracy loci.  Let $X$ be an algebraic scheme over
$\bk$ equipped with vector bundles $\cE_1, \dots, \cE_n$ of ranks
given by the dimension vector $e = (e_1, \dots, e_n)$.  Define the
bundle $\cV = \bigoplus_{a\in Q_1} \Hom(\cE_{t(a)},\cE_{h(a)})$ over
$X$.  Since the fibers of $\cV$ are isomorphic to the representation
space $V$, any quiver cycle $\Omega \subset V$ defines a translated
degeneracy locus $\wt \Omega \subset \cV$.  To be precise, let $\pi :
P \to X$ be the principal $\bG$-bundle such that $\cE_i = (E_i)_\bG =
P \times^\bG E_i$ for each $i$.  This bundle can be constructed as a
multi-frame bundle $P \subset \cE_1^{\oplus e_1} \oplus \dots \oplus
\cE_n^{\oplus e_n}$, with fibers $\pi^{-1}(x)$ consisting of lists of
bases of the fibers $\cE_i(x)$.  Then we have $\cV = V_\bG$ and $\wt
\Omega = \Omega_\bG \subset \cV$.

\begin{cor} \label{C:univlocus}
  The Grothendieck class of the translated degeneracy locus $\wt\Omega
  \subset \cV$ is given by
\[ [\cO_{\wt\Omega}] = \sum_\mu c_\mu(\Omega)\, 
   \groth_{\mu_1}(\cE_1-\cM_1) \cdots \groth_{\mu_n}(\cE_n-\cM_n) 
   \ \in K_\circ(\cV) \,,
\]
where $\cM_i = \bigoplus_{a:h(a)=i} \cE_{t(a)} = P \times^\bG M_i$.
Furthermore, the quiver coefficients for $\Omega$ are uniquely
determined by the truth of this identity for all non-singular
varieties $X$ and vector bundles $\cE_1,\dots,\cE_n$.
\end{cor}
\begin{proof}
  This follows from Proposition~\ref{P:gclass_unique} and the
  definition of quiver coefficients, since
  $\varphi_\bG(\groth_{\mu_i}(E_i - M_i)) = \groth_{\mu_i}(\cE_i -
  \cM_i)$.
\end{proof}

Define a {\em representation\/} $\cE_\bull$ of $Q$ on the vector
bundles $\cE_1,\dots,\cE_n$ over $X$ to be a collection of bundle maps
$\cE_{t(a)} \to \cE_{h(a)}$ corresponding to the arrows $a \in Q_1$.
Such a representation defines a section $s : X \to \cV$.  We define
the degeneracy locus $\Omega(\cE_\bull)$ as the scheme-theoretic
inverse image $\Omega(\cE_\bull) = s^{-1}(\wt\Omega) \subset X$.  This
degeneracy locus consists of all points in $X$ over which the bundle
maps of $\cE_\bull$ degenerate to representations in $\Omega$.  For
example, if $\wt\cE_\bull$ denotes the tautological representation of
$Q$ over $\cV$, defined by the universal maps between the pullbacks of
the vector bundles $\cE_i$ to $\cV$, then $\wt\Omega =
\Omega(\wt\cE_\bull)$.

Assume that $X$ has an action of an algebraic group $H$ over $\bk$ and
the representation $\cE_\bull$ consists of $H$-equivariant vector
bundles and bundle maps.  Then $P$ has a commuting $H$-action as in
section \ref{S:interpret} and $\cV$ is an $H$-vector bundle, so it
follows from Proposition~\ref{P:gclass_equiv} that the identity of
Corollary~\ref{C:univlocus} holds in $K_H(\cV)$.  It also follows that
$s : X \to \cV$ is an equivariant section.

We can define a localized class $\bOmega(\cE_\bull)$ in
$K_H(\Omega(\cE_\bull))$ by
\[ \bOmega(\cE_\bull) = s^!([\cO_{\wt\Omega}]) = \sum_{j\geq 0}
   (-1)^j [\Tor^\cV_j(\cO_X, \cO_{\wt \Omega})] \,.
\]
This definition is compatible with ($H$-equivariant) flat or regular
pullback and proper pushforward \cite{fulton.macpherson:categorical},
and the image of $\bOmega(\cE_\bull)$ in $K_H(X)$ is given by
\[ \bOmega(\cE_\bull) = s^*[\cO_{\wt\Omega}] = 
   \sum_\mu c_\mu(\Omega)\, \groth_{\mu_1}(\cE_1
   - \cM_1) \cdots \groth_{\mu_n}(\cE_n - \cM_n) \,.
\]
Furthermore, if $X$ and $\Omega$ are Cohen-Macaulay and the
codimension of $\Omega(\cE_\bull)$ in $X$ is equal to the codimension
of $\Omega$ in $V$, then we have $ \bOmega(\cE_\bull) =
[\cO_{\Omega(\cE_\bull)}] \in K_H(\Omega(\cE_\bull))$.  This is true
because a local regular sequence generating the ideal of $X$ in $\cV$
restricts to a local regular sequence defining the ideal of
$\Omega(\cE_\bull)$ in $\wt\Omega$ \cite[Lemma
A.7.1]{fulton:intersection}.  This implies that
$\Tor^\cV_j(\cO_X,\cO_{\wt\Omega}) = 0$ for all $j>0$, so
$\bOmega(\cE_\bull) = [\cO_X \otimes_{\cO_\cV} \cO_{\wt\Omega}] =
[\cO_{\Omega(\cE_\bull)}]$.  We note that if $Q$ is a Dynkin quiver of
type A or D and $\bk$ is algebraically closed, then any orbit closure
$\Omega \subset V$ is Cohen-Macaulay
\cite{lakshmibai.magyar:degeneracy, bobinski.zwara:normality,
  bobinski.zwara:schubert}.  The following corollary generalize all
the above formulas involving quiver coefficients, including
Definition~\ref{D:quivcoef}.

\begin{cor}
  \label{C:deglocus} 
  Let $\cE_\bull$ be a representation of $Q$ consisting of
  $H$-equivariant vector bundles and bundle maps over $X$.  Assume
  that both $X$ and $\Omega$ are Cohen-Macaulay and that the
  codimension of $\Omega(\cE_\bull)$ in $X$ is equal to the
  codimension of $\Omega$ in $V$.  Then we have
  \[ [\cO_{\Omega(\cE_\bull)}] = \sum_\mu c_\mu(\Omega)\,
     \groth_{\mu_1}(\cE_1-\cM_1) \cdots 
     \groth_{\mu_n}(\cE_n-\cM_n) \in K_H(X) \,.
  \]
\end{cor}

Let $X$ be a non-singular variety.  Subject to mild conditions,
corollaries \ref{C:univlocus} and \ref{C:deglocus} have cohomological
analogues.  For a partition $\lambda = (\lambda_1, \dots, \lambda_l)$
and vector bundles $\cA$ and $\cB$ over $X$, define $s_\lambda(\cA -
\cB) = \det( h_{\lambda_i+j-i} )_{l \times l} \in H^*(X)$, where the
classes $h_i$ are defined by $\sum_{i \geq 0} h_i =
c(\cB^\vee)/c(\cA^\vee)$, and $c(\cA^\vee) = 1 - c_1(\cA) + c_2(\cA) -
\cdots$ is the total Chern class of $\cA^\vee$.

\begin{cor}
  \label{C:cohom_univlocus}
  If $X$ admits an ample line bundle or if $Q$ is a quiver of Dynkin
  type, then the Chow class of the translated degeneracy locus
  $\wt\Omega \subset \cV$ is given by
  \[ [\wt\Omega] = \sum_{\sum|\mu_i|=\codim(\Omega)}
  c_\mu(\Omega)\, s_{\mu_1}(\cE_1-\cM_1)\cdots s_{\mu_n}(\cE_n-\cM_n)
  {\,\scriptstyle\cap\,} [\cV] \ \in H_*(\cV) \,.
  \]
  Without these conditions, this identity holds in $H^*(\cV) \otimes
  \Q$.
\end{cor}

If $X$ has an ample line bundle, then one can deduce this statement
from the expression for $[\Omega] \in H^*_\bG(V)$ along the lines of
\cite[\S 2.5]{buch.rimanyi:formula}, and if $Q$ is of Dynkin type,
then one can replace Grothendieck polynomials with Schur polynomials
in the proof of the formula for quiver coefficients given in section
\ref{S:formula}.  The formula with rational coefficients follows from
Corollary~\ref{C:univlocus} by using the Chern character
\cite[Ex.~15.2.16]{fulton:intersection}.  If $H$ is a linear algebraic
group, then a cohomological analogue of Corollary~\ref{C:deglocus} can
be proved from Corollary~\ref{C:cohom_univlocus} by first replacing
$X$ with the Borel construction $P \times^H X$, where $P/H$ is an
algebraic approximation of the classifying space of $H$
\cite{totaro:chow, edidin.graham:equivariant}, and then applying
\cite[Prop.~7.1]{fulton:intersection}.  We leave the details to the
reader.  We expect that Corollary~\ref{C:cohom_univlocus} is true
without the assumptions, but have not found a proof.


\section{Resolution of singularities}
\label{S:resolution}

Our formula for quiver coefficients of Dynkin type is based on
Reineke's resolution of the singularities of orbit closures for Dynkin
quivers \cite{reineke:quivers}.  It will be convenient to formulate
Reineke's construction for an arbitrary quiver $Q$, together with a
representation of $Q$ on vector bundles over a base scheme $X$.

Let $X$ be an algebraic scheme over $\bk$ equipped with a
representation $\cE_\bull$ of $Q$ on vector bundles over $X$, with
$\rank(\cE_i) = e_i$.  Let $i \in Q_0$ be a quiver vertex and let $r$
be an integer with $1 \leq r \leq e_i$.  Let $\rho : Y = \Gr(e_i-r,
\cE_i) \to X$ be the Grassmann bundle of rank $r$ quotients of
$\cE_i$, with universal exact sequence $0 \to \cS \to \cE_i \to \cQ
\to 0$.  (We will avoid explicit notation for pullback of vector
bundles.)  We define the scheme $X_{i,r} = X_{i,r}(\cE_\bull)$ to be
the zero scheme $X_{i,r} = Z(\cM_i \to \cQ) \subset Y$, where $\cM_i =
\bigoplus_{a:h(a)=i} \cE_{t(a)}$ and the map $\cM_i \to \cQ$ is
obtained by composing the projection $\cE_i \to \cQ$ with the sum of
the bundle maps $\cE_j \to \cE_i$ of the representation $\cE_\bull$.
This scheme has a natural projection $\rho : X_{i,r} \to X$.  Notice
that on $X_{i,r}$, all the maps $\cE_j \to \cE_i$ can be factored
through the subbundle $\cS \subset \cE_i$.  Using the factored maps,
we obtain an {\em induced representation\/} $\cE'_\bull$ over
$X_{i,r}$ on vector bundles given by $\cE'_j = \cE_j$ for $j\neq i$
and $\cE'_i = \cS$.

More generally, let $\bi = (i_1,\dots,i_m)$ be a sequence of quiver
vertices and $\br = (r_1,\dots,r_m)$ a sequence of positive integers,
such that for each $i \in Q_0$ we have $e_i \geq \sum_{i_j=i} r_j$.
We can iterate the above construction and define
\[ X_{\bi,\br} = X_{\bi,\br}(\cE_\bull) = (\cdots (( X_{i_1,r_1})_{i_2,r_2})
   \cdots )_{i_m,r_m} \,.
\]
The variety $(X_{i_1,r_1})_{i_2,r_2}$ is constructed using the induced
representation $\cE'_\bull$ on $X_{i_1,r_1}$, etc.  Let $\pi :
X_{\bi,\br} \to X$ denote the projection.  In general, this map may
have fibers of positive dimension.

Now let $Q$ be a quiver of Dynkin type and let $\Phi^+ \subset \N^n$
be the set of positive roots for the underlying Dynkin diagram.  Here
we identify the simple roots with the unit vectors $\epsilon_i \in
\N^n$, $1 \leq i \leq n$.  According to Gabriel's classification
\cite{gabriel:unzerlegbare}, there is a unique indecomposable
representation of $Q$ with dimension vector $\alpha$ for every
positive root $\alpha \in \Phi^+$, and all indecomposable
representations have this form.  This implies that the $\bG$-orbits in
$V$ correspond to sequences $(m_\alpha) \in \N^{\Phi^+}$ for which
$\sum m_\alpha \alpha$ is equal to the dimension vector $e$.
Furthermore, since the number of orbits is finite, it follows that
every quiver cycle in $V$ is an orbit closure.

For dimension vectors $\alpha, \beta \in \N^n$, let $\euler{\alpha,
  \beta} = \sum_{i=1}^n \alpha_i \beta_i - \sum_{a \in Q_1}
\alpha_{t(a)} \beta_{h(a)}$ denote the Euler form for $Q$.  Let $\Phi'
\subset \Phi^+$ be any subset of the positive roots.  A partition
$\Phi' = \cI_1 \cup \cdots \cup \cI_s$ of this set is called {\em
  directed\/} if $\euler{\alpha, \beta} \geq 0$ for all $\alpha,\beta
\in \cI_j$, $1 \leq j \leq s$, and $\euler{\alpha,\beta} \geq 0 \geq
\euler{\beta,\alpha}$ for all $\alpha \in \cI_i$ and $\beta \in \cI_j$
with $i<j$.  A directed partition always exists because the category
of representations of $Q$ is representation-directed
\cite{ringel:tame*1}.

Let $(m_\alpha) \in \N^{\Phi^+}$ be a sequence representing an orbit
closure $\Omega \subset V$, let $\Phi' \subset \Phi^+$ be a subset
containing $\{ \alpha : m_\alpha \neq 0 \}$, and let $\Phi' = \cI_1
\cup \dots \cup \cI_s$ be a directed partition.  For each $j \in
[1,s]$, write $\sum_{\alpha \in \cI_j} m_\alpha \alpha = (p^j_1,
\dots, p^j_n) \in \N^n$.  Then let $\bi^j = (i_1,\dots,i_l)$ be any
sequence of the vertices $i \in Q_0$ for which $p^j_i \neq 0$, with no
vertices repeated, and ordered so that the tail of any arrow of $Q$
comes before the head.  Set $\br^j = (p^j_{i_1}, \dots, p^j_{i_l})$.
Finally, let $\bi$ and $\br$ be the concatenated sequences $\bi =
\bi^1 \bi^2 \cdots \bi^s$ and $\br = \br^1 \br^2 \cdots \br^s$.  We
will call any pair of sequences $(\bi,\br)$ arising in this way for a
{\em resolution pair\/} for $\Omega$.

Let $\wt E_\bull$ denote the representation of $Q$ on the vector
bundles $\wt E_i = V \times E_i$ over $V$, defined by the tautological
maps $E_{t(a)} \to E_{h(a)}$, $(\phi, y) \mapsto (\phi, \phi_a(y))$,
for $a \in Q_1$.

\begin{thm}[Reineke]
  Let $Q$ be a quiver of Dynkin type, $\Omega \subset V$ an orbit
  closure, and $(\bi,\br)$ a resolution pair for $\Omega$.  Then the
  map $\pi : V_{\bi,\br}(\wt E_\bull) \to V$ has image $\Omega$ and is
  a birational isomorphism of $V_{\bi,\br}(\wt E_\bull)$ with
  $\Omega$.
\end{thm}

We note that Reineke's paper \cite{reineke:quivers} states this
theorem only in the case where the resolution pair $(\bi,\br)$ is
constructed from a directed partition of the set of all positive roots
$\Phi^+$, but the proof covers the more general statement.

Our formula for quiver coefficients given in the next section uses a
resolution pair $(\bi,\br)$ and requires a number of steps
proportional to the common length of $\bi$ and $\br$.  It is therefore
desirable to make these sequences as short as possible.  One
reasonable choice is to take the minimal set $\Phi' = \{ \alpha :
m_\alpha \neq 0 \}$ and use the following `greedy' algorithm to
produce a shortest possible directed partition of $\Phi'$.

Define $\cI(\Phi')$ to be the (unique) largest subset of $\Phi'$ for
which every element $\alpha$ in $\cI(\Phi')$ satisfies that
$\euler{\alpha,\beta} \geq 0$ for all $\beta \in \Phi'$, and
$\euler{\beta,\alpha} \leq 0$ for all $\beta \in \Phi' \smallsetminus
\cI(\Phi')$.  This set can be constructed by starting with all roots
$\alpha \in \Phi'$ for which the first inequality holds, and then
discarding roots until the second inequality is satisfied.  Since at
least one directed partition for $\Phi'$ exists, it follows that
$\cI(\Phi') \neq \emptyset$.  We now obtain a shortest possible
directed partition of $\Phi'$ by setting $\cI_1 = \cI(\Phi')$, $\cI_2
= \cI(\Phi' \smallsetminus \cI_1)$, $\cI_3 = \cI(\Phi' \smallsetminus
(\cI_1 \cup \cI_2))$, etc.

\begin{example} \label{E:inbound}
  Let $Q = \{ 1 \to 2 \ot 3 \}$ be the quiver of type A$_3$ in which
  both arrows point toward the center.  The set of positive roots is
  $\Phi^+ = \{ \alpha_{ij} \mid 1 \leq i < j \leq 3 \}$, where
  $\alpha_{ij} = \sum_{p=i}^j \varepsilon_p$.  Given an arbitrary
  partition $\Phi^+ = \cI_1 \cup \dots \cup \cI_s$, we write
  $\eta(\alpha) = j$ for $\alpha \in \cI_j$.  The partition is
  directed if and only if $\eta(\alpha) \leq \eta(\beta)$ whenever the
  following graph has an arrow from $\alpha$ to $\beta$, and
  $\eta(\alpha) < \eta(\beta)$ when the graph has a solid arrow from
  $\alpha$ to $\beta$.
  \[\xymatrix{ & \alpha_{12} \ar@{-->}[rd] \ar[rr] & & \alpha_{33} \\
    \alpha_{22} \ar[rr] \ar@{-->}[ru] \ar@{-->}[rd] & & 
    \alpha_{13} \ar@{-->}[ru] \ar@{-->}[rd] \\
    & \alpha_{23} \ar[rr] \ar@{-->}[ru] & & \alpha_{11}}
  \]
  This graph is constructed by drawing a solid arrow from $\alpha$ to
  $\beta$ if $\euler{\beta,\alpha} < 0$, and a dotted arrow from
  $\alpha$ to $\beta$ if $\euler{\beta,\alpha} \geq 0$ and
  $\euler{\alpha,\beta} > 0$.  The shortest directed partition of the
  positive roots is $\Phi^+ = \{ \alpha_{22}, \alpha_{12}, \alpha_{23}
  \} \cup \{ \alpha_{13}, \alpha_{11}, \alpha_{33} \}$.
  
  Let $\Omega \subset V = \Hom(E_1,E_2) \oplus \Hom(E_3,E_2)$ be an
  orbit closure, corresponding to the integer sequence $(m_{ij}) \in
  \N^{\Phi^+}$ with $\sum m_{ij} \alpha_{ij} = e = (e_1,e_2,e_3)$.
  Then $\Omega$ is defined set-theoretically by
   \begin{multline*}
    \Omega = \{ (\phi_1,\phi_3) \in V \mid 
    \rank(\phi_1) \leq m_{12}+m_{13} \text{ and }
    \rank(\phi_3) \leq m_{23}+m_{13} \\ \text{ and }
    \rank(\phi_1+\phi_3 : E_1\oplus E_3 \to E_2) \leq
    m_{12}+m_{23}+m_{13} \} \,.
  \end{multline*}
  As preparation for section~\ref{S:a3}, we will work out the
  desingularization of $\Omega$ obtained from the directed partition
  $\Phi^+ = \{ \alpha_{22} \} \cup \{ \alpha_{12}, \alpha_{23},
  \alpha_{13} \} \cup \{ \alpha_{11}, \alpha_{33} \}$.  The
  corresponding resolution pair $(\bi,\br)$ is given by $\bi = (2,
  1,3,2, 1,3)$ and $\br = (m_{22},
  m_{12}+m_{13},m_{23}+m_{13},e_2-m_{22}, m_{11},m_{33})$.  Form the
  product of Grassmann varieties $P = \Gr(m_{11},E_1) \times
  \Gr(e_2-m_{22},E_2) \times \Gr(m_{33},E_3)$.  The desingularization
  of $\Omega$ defined by $(\bi,\br)$ is the variety
  \[ V_{\bi,\br}(\wt E_\bull) = \{ (S_1,S_2,S_3,\phi_1,\phi_3) \in
     P\times V \mid \phi_i(E_i) \subset S_2 \text{ and } \phi_i(S_i) =
     0 \text{ for } i=1,3 \} \,. 
  \]
\end{example}


\section{A formula for Quiver coefficients}
\label{S:formula}

Let $Q$ be an arbitrary quiver, and let $X$ be an algebraic scheme
over $\bk$ equipped with vector bundles $\cE_1,\dots,\cE_n$ such that
$\rank(\cE_i) = e_i$ for each $i$.  Over the scheme $\cV =
\bigoplus_{a\in Q_1} \Hom(\cE_{t(a)},\cE_{h(a)})$ we have a
tautological representation $\wt\cE_\bull$ of $Q$ on (the pullbacks
of) the bundles $\cE_i$.  Any pair of sequences $\bi = (i_1,\dots,i_m)
\in Q_0^m$ and $\br = (r_1,\dots,r_m) \in \N^m$, with $\sum_{i_j=i}
r_j \leq e_i$ for each $i$, defines a map $\pi :
\cV_{\bi,\br}(\wt\cE_\bull) \to \cV$.  In this section we give a
formula for coefficients $c_\mu(\bi,\br) \in \Z$, indexed by sequences
of partitions $\mu = (\mu_1,\dots,\mu_n)$ with $\ell(\mu_i) \leq e_i$,
such that
\[ \pi_* [\cO_{\cV_{\bi,\br}}] = \sum_\mu c_\mu(\bi,\br)\,
   \groth_{\mu_1}(\cE_1-\cM_1)\, \groth_{\mu_2}(\cE_2-\cM_2) \cdots
   \groth_{\mu_n}(\cE_n-\cM_n) \ \in K_\circ(\cV) \,,
\]
where $\pi_* : K_\circ(\cV_{\bi,\br}) \to K_\circ(\cV)$ is the proper
pushforward along $\pi$.  If $Q$ is a quiver of Dynkin type and
$(\bi,\br)$ is a resolution pair for an orbit closure $\Omega \subset
V$ with rational singularities, then $c_\mu(\Omega) = c_\mu(\bi,\br)$.
Our formula is stated in terms of operators on tensors of Grothendieck
polynomials which we proceed to define.

Let $i \in Q_0$ be a quiver vertex.  We let $\psi_i : \Gamma^{\otimes
  n+1} \to \Gamma^{\otimes n+1}$ denote the linear operator which
applies the coproduct $\Delta$ to the $i$-th factor and multiplies one
of the components of this coproduct to the last factor.  More
precisely, $\psi_i$ is defined by
\begin{multline*}
  \psi_i(\groth_{\mu_1} \otimes \dots \otimes \groth_{\mu_n} 
         \otimes \groth_{\lambda}) = \\
  \sum_{\sigma,\nu} \left( 
  \sum_\tau d^{\mu_i}_{\sigma \tau}\, c^\nu_{\tau \lambda} \right)\,
  \groth_{\mu_1} \otimes \dots \otimes \groth_{\mu_{i-1}} \otimes
  \groth_\sigma \otimes \groth_{\mu_{i+1}} \otimes \dots \otimes
  \groth_{\mu_n} \otimes \groth_\nu \,,
\end{multline*}
where the sum is over all partitions $\sigma$, $\tau$, and $\nu$, and
the constants $d^{\mu_i}_{\sigma \tau}$ and $c^\nu_{\tau \lambda}$ are
defined in section~\ref{S:groth}.

For integers $r,c$ with $r \geq 0$, define the linear map
$\cA_{i,r\times c} : \Gamma^{\otimes n+1} \to \Gamma^{\otimes n}$ by
\[ \cA_{i,r\times c}(\groth_{\mu_1}\otimes \dots \otimes
   \groth_{\mu_n} \otimes \groth_\nu) =
   \groth_{\mu_1} \otimes \dots \otimes \groth_{\mu_{i-1}} \otimes
   \groth_{(c)^r+\nu,\mu_i} \otimes \groth_{\mu_{i+1}} \otimes
   \dots \otimes \groth_{\mu_n}
\]
if $\ell(\nu) \leq r$, and $\cA_{i,r\times c}(\groth_{\mu_1} \otimes
\dots \otimes \groth_{\mu_n} \otimes \groth_\nu) = 0$ otherwise.  Here
$(c)^r+\nu,\mu_i$ denotes the concatenation of the integer sequence
$(c+\nu_1,\dots,c+\nu_r)$ with the partition $\mu_i$.  When this does
not result in a partition, then the Grothendieck polynomial
$\groth_{(c)^r+\nu,\mu_i}$ is defined by equation (\ref{E:intseq}).
The operator $\cA_{i,r\times c}$ will be applied with negative as well
as positive integers $c$.

Let $a_1, \dots, a_l \in Q_1$ be the arrows starting at $i$, i.e.\ 
$t(a_j) = i$ for each $j$.  Define the linear map $\Phi_{i,r}^{Q,e} :
\Gamma^{\otimes n} \to \Gamma^{\otimes n}$ by
\[ \Phi_{i,r}^{Q,e}(P) = \cA_{i,r\times c}\, \psi_{h(a_1)} \, \cdots
\, \psi_{h(a_l)} (P \otimes 1)
\]
where $c = \rank(\cM_i) - e_i + r$.

Given sequences $\bi = (i_1,\dots,i_m) \in Q_0^m$ and $\br =
(r_1,\dots,r_m) \in \N^m$ as above, we define a tensor
$P^{Q,e}_{\bi,\br} \in \Gamma^{\otimes n}$ as follows.  If $m=0$, then
we set $P^{Q,e}_{\bi,\br} = 1 \otimes \dots \otimes 1$.  Otherwise we
may assume by induction that $P^{Q,e'}_{\bi',\br'} \in \Gamma^{\otimes
  n}$ has already been defined, where $\bi' = (i_2,\dots,i_m)$ and
$\br' = (r_2,\dots,r_m)$, and $e'$ is the dimension vector defined by
$e'_j = e_j$ for $j \neq i_1$ and $e'_{i_1} = e_{i_1}-r_1$.  In this
case we set $P^{Q,e}_{\bi,\br} =
\Phi^{Q,e}_{i_1,r_1}(P^{Q,e'}_{\bi',\br'})$.  We define the
coefficients $c_\mu(\bi,\br)$ as the coefficients in the expansion
\[ P^{Q,e}_{\bi,\br} = \sum_\mu c_\mu(\bi,\br)\, \groth_{\mu_1}\otimes
   \groth_{\mu_2}\otimes \dots \otimes \groth_{\mu_n} \,.
\]
It follows from this definition that $c_\mu(\bi,\br)$ is zero unless
$\ell(\mu_i) \leq e_i$ for each $i$.

Given any element $P = \sum c_\mu\, \groth_{\mu_1} \otimes \dots
\otimes \groth_{\mu_n} \in \Gamma^{\otimes n}$ and
$\alpha_1,\dots,\alpha_n \in K^\circ(X)$, we set
$P(\alpha_1,\dots,\alpha_n) = \sum c_\mu\, \groth_{\mu_1}(\alpha_1)\,
\groth_{\mu_2}(\alpha_2) \cdots \groth_{\mu_n}(\alpha_n) \in
K^\circ(X)$.  The following theorem gives the geometric interpretation
of the coefficients $c_\mu(\bi,\br)$.

\begin{thm} \label{T:quivcoef}
  Let $\pi : \cV_{\bi,\br}(\wt\cE_\bull) \to \cV$ be the map
  associated to sequences $\bi, \br$.  Then
  $\pi_*([\cO_{\cV_{\bi,\br}}]) = P^{Q,e}_{\bi,\br}(\cE_1-\cM_1,
  \dots, \cE_n-\cM_n) \in K^\circ(\cV)$.
\end{thm}

\begin{cor} \label{C:dynkincoef}
  Let $Q$ be a quiver of Dynkin type, $\Omega \subset V$ an orbit
  closure, and $(\bi,\br)$ a resolution pair for $\Omega$.  If
  $\Omega$ has rational singularities then $P_\Omega =
  P^{Q,e}_{\bi,\br}$, or equivalently, the quiver coefficients of
  $\Omega$ are given by $c_\mu(\Omega) = c_\mu(\bi,\br)$.
  Furthermore, this identity is true for all cohomological quiver
  coefficients, without the assumption about rational singularities.
\end{cor}
\begin{proof}
  If $X$ is a non-singular variety, then it follows from Reineke's
  theorem that $\pi : \cV_{\bi,\br}(\wt\cE_\bull) \to \wt\Omega$ is a
  desingularization of the translated degeneracy locus $\wt\Omega
  \subset \cV$.  If $\Omega$ has rational singularities, then
  $\pi_*([\cO_{\cV_{\bi,\br}}]) = [\cO_{\wt\Omega}] \in K_\circ(\cV)$,
  so the corollary follows by comparing Theorem~\ref{T:quivcoef} to
  Corollary~\ref{C:univlocus}.  Without this assumption, we still have
  $\pi_* [\cV_{\bi,\br}] = [\wt\Omega]$ in the Chow ring of $\cV$,
  which suffices to determine the cohomological quiver coefficients.
\end{proof}

\begin{remark}
  If $\Omega \subset V$ is an orbit closure of Dynkin type, then the
  quiver coefficients for $\Omega$ are identical to the quiver
  coefficients for $\wb \Omega = \Omega \times_{\Spec(\bk)} \Spec(\wb
  \bk)$, where $\wb \bk$ is an algebraic closure of $\bk$.
  Corollary~\ref{C:dynkincoef} therefore applies also if $\wb \Omega$
  has rational singularities, which has been proved for quivers of
  type A in any characteristic and for quivers of type D in
  characteristic zero \cite{lakshmibai.magyar:degeneracy,
    bobinski.zwara:normality, bobinski.zwara:schubert}.
\end{remark}

We have computed the coefficients $c_\mu(\bi,\br)$ for lots of
randomly chosen quivers $Q$ and sequences $\bi$ and $\br$, and in all
cases they had alternating signs in the following sense.

\begin{conj} \label{C:altformula}
  We have $(-1)^{\sum |\mu_i| + \sum |\mu'_i|}\, c_\mu(\bi,\br)\,
  c_{\mu'}(\bi,\br) \geq 0$ for arbitrary sequences of partitions
  $\mu$ and $\mu'$.
\end{conj}

In almost all examples that we computed, the coefficients
$c_\mu(\bi,\br)$ of lowest degree were positive.  However, we also
found examples where the lowest degree coefficients were negative, the
next degree up were positive, etc.  We speculate that in many
examples, the class $\pi_*([\cO_{\cV_{\bi,\br}}])$ has been equal to
the Grothendieck class of the image of $\pi$, which is always a quiver
cycle in $V$.  We therefore regard our verification of
Conjecture~\ref{C:altformula} as additional evidence for
Conjecture~\ref{C:altsign}.  For the proof of
Theorem~\ref{T:quivcoef}, we need the following Gysin formula from
\cite[Thm.~7.3]{buch:grothendieck}.

\begin{thm} \label{T:gysin} Let $\cF$ and $\cB$ be vector bundles on
  $X$.  Write $\rank(\cF) = s+q$ and let $\rho : \Gr(s,\cF) \to X$ be
  the Grassmann bundle of $s$-planes in $\cF$ with universal exact
  sequence $0 \to \cS \to \rho^* \cF \to \cQ \to 0$.  Let $I =
  (I_1,\dots,I_q)$ and $J = (J_1,J_2,\dots)$ be finite sequences of
  integers such that $I_j \geq \rank(\cB)$ for all $j$.  Then
  \[ \rho_*( \groth_I(\cQ - \rho^* \cB) \cdot \groth_J(\cS - \rho^*
  \cB)) = \groth_{I-(s^q),J}(\cF - \cB) \ \in K_\circ(X) \,,
  \]
  where $I-(s^q),J = (I_1-s,\dots,I_q-s,J_1,J_2,\dots)$.
\end{thm}

Consider a variety $\cV_{i,r} = Z(\cM_i \to \cQ) \subset Y =
\Gr(e_i-r,\cE_i)$ as in the previous section, where $0 \to \cS \to
\cE_i \to \cQ \to 0$ is the universal exact sequence on $Y$.  Let
$\rho : \cV_{i,r} \to \cV$ be the projection and let $\cE'_\bull$ be
the induced representation on $\cV_{i,r}$.

\begin{lemma} \label{L:psi} Let $P' \in \Gamma^{\otimes n+1}$ and set
  $P = \psi_i(P')$.  Then $P'(\alpha_1,\dots,\alpha_n,\cQ) =
  P(\alpha_1,\dots,\alpha_{i-1},\alpha_i-\cQ,\alpha_{i+1},
  \dots,\alpha_n,\cQ)$ for any elements $\alpha_1,\dots,\alpha_n \in
  K^\circ(\cV_{i,r})$.
\end{lemma}
\begin{proof}
  For partitions $\mu_i$ and $\lambda$ we have
  $\groth_{\mu_i}(\alpha_i)\cdot \groth_\lambda(\cQ) =
  \groth_{\mu_i}(\alpha_i-\cQ+\cQ)\cdot \groth_{\lambda}(\cQ) =
  \sum_{\sigma,\tau} d^{\mu_i}_{\sigma\tau}\,
  \groth_\sigma(\alpha_i-\cQ) \cdot \groth_\tau(\cQ)\cdot
  \groth_\lambda(\cQ) = \sum_{\sigma,\tau} d^{\mu_i}_{\sigma\tau}\,
  \groth_\sigma(\alpha_i-\cQ) \sum_\nu c^\nu_{\tau\lambda}\,
  \groth_\nu(\cQ)$.
\end{proof}

\begin{prop} \label{P:Phi}
  Let $P' \in \Gamma^{\otimes n}$ and set $P = \Phi_{i,r}^{Q,e}(P)$
  and $\cM'_i = \bigoplus_{h(a)=i} \cE'_{t(a)}$.  Then $\rho_*(
  P'(\cE'_1-\cM'_1, \dots, \cE'_n-\cM'_n)) = P(\cE_1-\cM_1, \dots,
  \cE_n-\cM_n)$ in $K_\circ(\cV)$.
\end{prop}
\begin{proof}
  For each $j \in Q_0$ we have $[\cM_j] = [\cM'_j] + p [\cQ] \in
  K^\circ(\cV_{i,r})$, where $p$ is the number of arrows from $i$ to
  $j$.  Lemma~\ref{L:psi} therefore implies that $P'(\cE'_1-\cM'_1,
  \dots, \cE'_n-\cM'_n) = P''(\cE'_1-\cM_1,\dots,\cE'_n-\cM_n,\cQ)$
  where $P'' = \psi_{h(a_1)} \cdots \psi_{h(a_l)}(P' \otimes 1)$.
    
  It follows from Example~\ref{E:porteous} that $[\cO_{\cV_{i,r}}] =
  \groth_{R}(\cQ - \cM_i)$ in $K_\circ(Y)$, where $R =
  (\rank(\cM_i)^r)$.  The pushforward of $P'(\cE'_1-\cM'_1, \dots,
  \cE'_n-\cM'_n)$ from $\cV_{i,r}$ to $Y$ is therefore equal to
  $P''(\cE'_1-\cM_1, \dots, \cE'_n-\cM_n, \cQ) \cdot \groth_{R}(\cQ)$.
  
  Let $\mu_i$ and $\nu$ be partitions.  If $\ell(\nu) > r$ then
  $\groth_\nu(\cQ) = 0$.  Otherwise it follows from the factorization
  formula (\ref{E:factor}) that $\groth_\nu(\cQ) \groth_R(\cQ-\cM_i) =
  \groth_{R+\nu}(\cQ-\cM_i)$, and Theorem~\ref{T:gysin} implies that
  $\rho'_*(\groth_{R+\nu}(\cQ-\cM_i)\cdot \groth_{\mu_i}(\cS - \cM_i)) =
  \groth_{(c)^r+\nu,\mu_i}(\cE_i-\cM_i)$, where $\rho' : Y \to \cV$ is
  the projection and $c = \rank(\cM_i) - e_i + r$.  We conclude that
  $\rho_*(P'(\cE'_1-\cM'_1, \dots, \cE'_n-\cM'_n)) = P(\cE_1-\cM_1,
  \dots, \cE_n-\cM_n)$ where $P = \cA_{i,r\times c}(P'') =
  \Phi^{Q,e}_{i,r}(P')$.
\end{proof}

\begin{proof}[Proof of Theorem~\ref{T:quivcoef}]
  Let $X' = \Gr(e_{i_1}\!-{r_1},\cE_{i_1}) \to X$ be the Grassmann
  bundle of rank ${r_1}$ quotients of $\cE_{i_1}$.  Then the bundles
  $\cE'_j$ are defined on $X'$, and $Y = \cV \times_X X'$ can be
  constructed as the bundle $\bigoplus_{a \in Q_1}
  \Hom_{\cO_{X'}}(\cE_{t(a)},\cE_{h(a)})$ over $X'$.  It follows that
  $\cV_{{i_1},{r_1}} = Z(\cM_i \to \cE_i/\cE'_i) \subset Y$ is
  isomorphic to the bundle $\bigoplus_{a \in Q_1}
  \Hom_{\cO_{X'}}(\cE_{t(a)}, \cE'_{h(a)})$, which implies that
  $\cV_{{i_1},{r_1}}$ is an affine bundle over $\cV' = \bigoplus_{a
    \in Q_1} \Hom_{\cO_{X'}}(\cE'_{t(a)}, \cE'_{h(a)})$.  We
  furthermore have a fiber square:
  \[\xymatrix{
    \cV_{\bi,\br} \ar[r] \ar[d]_\beta &
    \cV'_{\bi',\br'}(\cE'_\bull)
    \ar[d]^{\beta'} \\
    \cV_{{i_1},{r_1}} \ar[r] & \cV' }
  \]
  By induction on $m$ we know that $\beta'_*(1) =
  P^{Q,e'}_{\bi',\br'}(\cE'_1-\cM'_1, \dots, \cE'_n-\cM'_n) \in
  K_\circ(\cV')$, and since the horizontal maps are flat, this implies
  that $\beta_*([\cO_{\cV_{\bi,\br}}]) = \beta_*(1) =
  P^{Q,e'}_{\bi',\br'}(\cE'_1-\cM'_1, \dots, \cE'_n-\cM'_n) \in
  K_\circ(\cV_{i_1,r_1})$.  Proposition~\ref{P:Phi} finally
  shows that $\pi_*([\cO_{\cV_{\bi,\br}}]) =
  \rho_*(P^{Q,e'}_{\bi',\br'}(\cE'_1-\cM'_1, \dots, \cE'_n-\cM'_n)) =
  P^{Q,e}_{\bi,\br}(\cE_1-\cM_1, \dots, \cE_n-\cM_n) \in
  K_\circ(\cV)$, as required.
\end{proof}

\begin{remark}
  For applications of our formula, it would be useful to know the
  reduced equations generating the ideal of an orbit closure $\Omega
  \subset V$ for a quiver $Q$ of Dynkin type.  For example, such
  equations will result in a more explicit construction of the
  degeneracy loci $\Omega(\cE_\bull)$ defined by $\Omega$.
  
  Let $\phi \in V$ be a representation of $Q$ on the vector spaces
  $E_1, \dots, E_n$, and fix another representation $\psi =
  (\psi_a)_{a \in Q_1}$ on vector spaces $F_1, \dots, F_n$.  A
  homomorphism from $\psi$ to $\phi$ is a collection $\beta$ of linear
  maps $\beta_i : F_i \to E_i$ such that $\phi_a \beta_{t(a)} =
  \beta_{h(a)} \psi_a$ as a map from $F_{t(a)}$ to $E_{h(a)}$ for all
  $a \in Q_1$.  Let $\Hom(\psi,\phi)$ denote the vector space of all
  such homomorphisms.  Bongartz has proved in
  \cite[Prop.~3.2]{bongartz:on} that $\phi'$ belongs to the
  orbit closure $\Omega = \wb{\bG.\phi}$ if and only if $\dim
  \Hom(\psi,\phi') \geq \dim \Hom(\psi,\phi)$ for all (indecomposable)
  representations $\psi$ of $Q$.  Set $A = \bigoplus_{i \in Q_0}
  \Hom(F_i, E_i)$ and $B = \bigoplus_{a \in Q_1} \Hom(F_{t(a)},
  E_{h(a)})$, and let $\gamma_{\psi,\phi} : A \to B$ be the linear map
  given by $\gamma_{\psi,\phi}(\beta) = ( \beta_{h(a)} \psi_a - \phi_a
  \beta_{t(a)} )_{a \in Q_1}$.  Define $\rank_\psi(\phi) =
  \rank(\gamma_{\psi,\phi})$.  We then have
  \begin{equation} \label{E:ideal}
    \Omega = \{ \phi' \in V \mid \rank_\psi(\phi') \leq
    \rank_\psi(\phi) 
    \text{ $\forall$ indecomp.\ representations $\psi$ of $Q$} \} \,.
  \end{equation}
  This description of the orbit closure $\Omega$ gives rise to
  set-theoretic equations for $\Omega$ in terms of minors of the
  matrices $\gamma_{\psi,\phi}$.  It is interesting to ask if these
  equations in fact generate the ideal $I(\Omega) \subset k[V]$.  This
  has been proved for equioriented quivers of type A by Lakshmibai and
  Magyar \cite{lakshmibai.magyar:degeneracy}, but reduced equations
  for orbit closures appear to be unknown for quivers of other types.
  We have used Macaulay 2 \cite{grayson.stillman:macaulay} to check
  that minors of the matrices $\gamma_{\psi,\phi}$ in fact generate
  the ideal of the inbound A$_3$-orbit closure given by $m_{ij}=1$ for
  $1 \leq i < j \leq 3$ (see Example~\ref{E:inbound}).
  
  If $\cE_\bull$ is a representation of $Q$ on vector bundles over
  $X$, then each fixed representation $\psi$ of $Q$ defines a vector
  bundle map from $\cA = \bigoplus_{i \in Q_0} \Hom(F_i \otimes \cO_X,
  \cE_i)$ to $\cB = \bigoplus_{a \in Q_1} \Hom(F_{t(a)}\otimes \cO_X,
  \cE_{h(a)})$, and the degeneracy locus $\Omega(\cE_\bull)$ is the
  set of points $x \in X$ where the rank of this bundle map is at most
  $\rank_\psi(\phi)$ for all $\psi$.  Assuming that (\ref{E:ideal})
  gives the reduced equations of $\Omega$, this description of
  $\Omega(\cE_\bull)$ also captures its scheme structure.
\end{remark}


\section{Quiver coefficients of type A$_3$}
\label{S:a3}

In this section we prove combinatorial formulas for the
(non-equioriented) quiver coefficients of type A$_3$.  These formulas
are based on counting set-valued tableaux, and show that the
coefficients have alternating signs.

\subsection{Inbound A$_3$ quiver}

Let $Q = \{ 1 \to 2 \ot 3 \}$ be the inbound quiver of type A$_3$ from
Example~\ref{E:inbound}, and let $\Omega \subset V$ be the orbit
closure given by $(m_{ij}) \in \N^{\Phi^+}$.  For partitions
$\lambda$, $\mu$, and $\nu$, define the coefficient
\[ c_{\lambda,\mu,\nu} = \sum_{\sigma,\tau}    
   d^{(m_{33})^{m_{12}}}_{\lambda,\sigma} \,
   d^{(m_{11})^{m_{23}}}_{\tau,\nu} \, 
   c^\mu_{\sigma\tau} \,,
\]
where the sum is over all partitions $\sigma$ and $\tau$.

\begin{prop} 
  The coefficient $c_{\lambda,\mu,\nu}$ is equal to
  $(-1)^{|\lambda|+|\mu|+|\nu|-m_{33}m_{12}-m_{11}m_{23}}$ times the
  number of pairs $(\sigma,T)$ of a partition $\sigma$ contained in
  the rectangle $(m_{33})^{m_{12}}$ with $m_{12}$ rows and $m_{33}$
  columns, and a set-valued tableau $T$ whose shape is a partition
  contained in $(m_{11})^{m_{23}}$, satisfying the following
  conditions.

\begin{romenum}
\item If $\sigma$ is placed in the upper-left corner of the rectangle
  $(m_{33})^{m_{12}}$ and the 180 degree rotation of $\lambda$ is
  placed in the lower-right corner, then their union is the whole
  rectangle and their overlap is a rook-strip, i.e.\ the the overlap
  contains at most one box in any row or column.
  
\item If $T$ is placed in the upper-left corner of the rectangle
  $(m_{11})^{m_{23}}$ and the 180 degree rotation of $\nu$ is placed
  in the lower-right corner, then their union is the whole rectangle
  and their overlap is a rook-strip.
  
\item The composition $w(T)u(\sigma)$ is a reverse lattice word with
  content $\mu$ (with the terminology of Theorem~\ref{T:mult}.)
\end{romenum}
\end{prop}
\begin{proof}
This follows from Theorem~\ref{T:mult} because
$d^{(m_{33})^{m_{12}}}_{\lambda,\sigma}$ is non-zero exactly when the
condition (i) is satisfied, in which case
$d^{(m_{33})^{m_{12}}}_{\lambda,\sigma} =
(-1)^{|\lambda|+|\sigma|-m_{33}m_{12}}$.  Notice also that (i) and
(ii) can only be satisfied if $\lambda \subset (m_{33})^{m_{12}}$ and
$\nu \subset (m_{11})^{m_{23}}$.
\end{proof}

\begin{thm} \label{T:inbound}
The quiver coefficients of the inbound quiver of type A$_3$ are given
by 
\[ P_\Omega = \sum_{\lambda,\mu,\nu} c_{\lambda,\mu,\nu}\, 
   \groth_\lambda \otimes \groth_{(m_{11}+m_{13}+m_{33})^{m_{22}},\mu}
   \otimes \groth_\nu \,.
\]
\end{thm}

\begin{lemma} \label{L:dualgysin}
  In the situation of Theorem~\ref{T:gysin}, let $\lambda$ be a
  partition such that $\lambda_1 = \lambda_b = s$, where $b =
  \rank(\cB)$.  Then $\rho_*(\groth_\lambda(\rho^* \cB - \cS)) =
  \groth_{(\lambda_{q+1},\lambda_{q+2},\dots)}(\cB - \cF)$.
\end{lemma}
\begin{proof}
  The Grassmann bundle $\Gr(s,\cF)$ of $s$-planes in $\cF$ is
  identical to the bundle $\Gr(q,\cF^\vee)$ of $q$-planes in
  $\cF^\vee$, with tautological exact sequence $0 \to \cQ^\vee \to
  \rho^* \cF^\vee \to \cS^\vee \to 0$.  The lemma follows from
  Theorem~\ref{T:gysin} by using the identity $\groth_{\lambda}(\rho^*
  \cB - \cS) = \groth_{\lambda'}(\cS^\vee - \rho^* \cB^\vee)$.
\end{proof}

\begin{proof}[Proof of Theorem \ref{T:inbound}]
  Let $X$ be a smooth variety with vector bundles $\cE_1, \cE_2,
  \cE_3$ of ranks $e_1,e_2,e_3$, and let $\wt\Omega \subset \cV =
  \Hom(\cE_1,\cE_2) \oplus \Hom(\cE_3,\cE_2)$ be the translated
  degeneracy locus.  Form the product of Grassmann bundles
\[ P =
  \Gr(m_{11}, \cE_1) \times_\cV \Gr(e_2-m_{22}, \cE_2) \times_\cV
  \Gr(m_{33}, \cE_3) \xrightarrow{\ \pi\ } \cV
\]
  with tautological subbundles $\cE'_i
  \subset \cE_i$, $1 \leq i \leq 3$.  The desingularization of
  $\wt\Omega$ is the iterated zero section $\cV_{\bi,\br} = Z(\cE'_1
  \oplus \cE'_3 \to \cE'_2) \subset Z(\cE_1 \oplus \cE_3 \to
  \cE_2/\cE'_2) \subset P$.  The Thom-Porteous formula
  (Example~\ref{E:porteous}) implies that the Grothendieck class of
  this locus in $K_\circ(P)$ is given by 
  \[ [\cO_{\cV_{\bi,\br}}] =
    \groth_{(m_{11})^{e_2-m_{22}}}(\cE'_2 - \cE'_1)\,
    \groth_{(m_{33})^{e_2-m_{22}}}(\cE'_2 - \cE'_3)\,
    \groth_{(e_1+e_2)^{m_{22}}}(\cE_2/\cE'_2 - \cE_1\oplus \cE_3) \,.
  \]
  The pushforward of this class along the projection $P \to
  \Gr(e_2-m_{22},\cE_2)$ is equal to
  $\groth_{(m_{11})^{m_{23}}}(\cE'_2 - \cE_1)\,
  \groth_{(m_{33})^{m_{12}}}(\cE'_2 - \cE_3)\,
  \groth_{(e_1+e_3)^{m_{22}}}(\cE_2/\cE'_2 - \cE_1\oplus \cE_3)$ by
  Lemma~\ref{L:dualgysin}.  The first two factors of this product can
  be rewritten as
\[\begin{split} 
&  \groth_{(m_{11})^{m_{23}}}(\cE'_2 - \cE_1)\,
   \groth_{(m_{33})^{m_{12}}}(\cE'_2 - \cE_3) 
\\
&= \sum_{\lambda,\sigma,\tau,\nu}
   d^{(m_{33})^{m_{12}}}_{\lambda,\sigma} \,
   d^{(m_{11})^{m_{23}}}_{\tau,\nu} \, 
   \groth_\lambda(\cE_1)\, \groth_\sigma(\cE'_2 - \cE_1 \oplus \cE_3)\, 
   \groth_\tau(\cE'_2 - \cE_1 \oplus \cE_3)\, \groth_\nu(\cE_3)
\\
&= \sum_{\lambda,\mu,\nu} c_{\lambda,\mu,\nu}\,
   \groth_\lambda(\cE_1)\, \groth_\mu(\cE'_2 - \cE_1 \oplus \cE_3)\,
   \groth_\nu(\cE_3) \,.
\end{split}\]
Theorem~\ref{T:gysin} applied to the bundle 
$\Gr(e_2-m_{22},\cE_2) \to \cV$ therefore shows that
\[\begin{split} \pi_*([\cO_{\cV_{\bi,\br}}]) 
&= \sum_{\lambda,\mu,\nu} c_{\lambda,\mu,\nu}\,
   \groth_\lambda(\cE_1)\,
   \groth_{(m_{11}+m_{13}+m_{33})^{m_{22}},\mu}(\cE_2 - \cE_1\oplus \cE_3)\,
   \groth_\nu(\cE_3)
\end{split}\]
in $K_\circ(\cV)$, as required.
\end{proof}

\subsection{Outbound A$_3$ quiver}

Now let $Q = \{1 \ot 2 \to 3\}$ be the quiver of type A$_3$ with both
arrows pointing away from the center, and let $\Omega \subset V$ be
the orbit closure corresponding to the sequence $(m_{ij}) \in
\N^{\Phi^+}$, where $\Phi^+ = \{ \alpha_{ij} \mid 1 \leq i < j \leq
3\}$.  Let $R = (m_{22})^{m_{13}}$ be the rectangle with $m_{13}$ rows
and $m_{22}$ columns.  For partitions $\lambda, \mu, \nu$, we let
$d^R_{\lambda,\mu,\nu}$ denote the 2-fold coproduct coefficients
defined by $\Delta^2(\groth_R) = \sum_{\lambda,\mu,\nu}
d^R_{\lambda,\mu,\nu}\, \groth_\lambda \otimes \groth_\mu \otimes
\groth_\nu$.

\begin{prop}
  The coefficient $d^R_{\lambda,\mu,\nu}$ is zero unless $\lambda$,
  $\mu$, and $\nu$ are contained in $R$, in which case it is equal to
  $(-1)^{|\lambda|+|\mu|+|\nu|-m_{22}m_{13}}$ times the number of
  triples $(\sigma,\tau,T)$, where $\sigma$ and $\tau$ are partitions
  such that $\sigma \subset \tau \subset R$, and $T$ is a set-valued
  tableau of skew shape $\tau/\sigma$, satisfying the following
  conditions.
  
  \begin{romenum}
  \item
    The Young diagram $\sigma$ is contained in $\lambda$, and
    $\lambda/\sigma$ is a rook-strip.

  \item
    If $\tau$ is placed in the upper-left corner of $R$ and the 180
    degree rotation of $\nu$ is placed in the lower-right corner,
    then their union is $R$ and their overlap is a rook-strip.
    
  \item 
    The word $w(T)$ is a reverse lattice word with content $\mu$.
  \end{romenum}
\end{prop}
\begin{proof}
  It follows from \cite[Lemma~6.1]{buch:littlewood-richardson} that
  $\Delta^2(\groth_R) = \sum (-1)^{|\lambda|+ |\tau/\sigma|+
    |\nu|-|R|}\, \groth_\lambda \otimes \groth_{\tau/\sigma} \otimes
  \groth_\nu$, where the sum is over all partitions
  $\lambda,\sigma,\tau,\mu \subset R$ satisfying (i) and (ii).  The
  coefficient of $\groth_\mu$ in $\groth_{\tau/\sigma}$ is equal to
  $(-1)^{|\mu|-|\tau/\sigma|}$ times the number of set-valued tableaux
  $T$ of shape $\tau/\sigma$ satisfying (iii) by
  \cite[Thm.~6.9]{buch:littlewood-richardson}.
\end{proof}

\begin{thm} \label{T:outbound}
The quiver coefficients of the outbound quiver of type A$_3$ are given
by
\[ P_\Omega = \sum_{\lambda,\mu,\nu} d^R_{\lambda,\mu,\nu}\,
   \groth_{(m_{22}+m_{23})^{m_{11}},\lambda} \otimes 
   \groth_\mu \otimes
   \groth_{(m_{22}+m_{12})^{m_{33}},\nu} \,.
\]
\end{thm}
\begin{proof}
  We use the directed partition $\Phi^+ = \{ \alpha_{11} \} \cup \{
  \alpha_{33}, \alpha_{23}, \alpha_{13} \} \cup \{ \alpha_{22},
  \alpha_{12} \}$, and resolution pair $\bi = (1, 2,1,3, 2,1)$ and
  $\br = (m_{11}, m_{23}+m_{13}, m_{13}, e_3, m_{22}+m_{12}, m_{12})$.
  Given a non-singular variety $X$ with vector bundles
  $\cE_1,\cE_2,\cE_3$ of ranks $e_1,e_2,e_3$, form the product $P =
  \Fl(m_{12}, m_{12}+m_{13}; \cE_1) \times_{\cV} \Gr(m_{22}+m_{12},
  \cE_2) \to \cV$, with universal subbundles $\cE''_1 \subset \cE'_1
  \subset \cE_1$ and $\cE'_2 \subset \cE_2$.  The desingularization of
  $\wt\Omega \subset \cV$ corresponding to $(\bi,\br)$ is the iterated
  zero section $\cV_{\bi,\br} = Z(\cE'_2 \to \cE'_1/\cE''_1 \oplus
  \cE_3) \subset Z(\cE_2 \to \cE_1/\cE'_1) \subset P$.  The
  Grothendieck class of this locus in $K_\circ(P)$ is
  \[ [\cO_{\cV_{\bi,\br}}] =
     \groth_{(e_2)^{m_{11}}}(\cE_1/\cE'_1 - \cE_2)\,
     \groth_{(m_{22}+m_{12})^{m_{13}}}(\cE'_1/\cE''_1 - \cE'_2)\,
     \groth_{(m_{22}+m_{12})^{e_3}}(\cE_3 - \cE'_2) \,, 
  \]
  and by
  Theorem~\ref{T:gysin}, the pushforward of this class along the
  projection $P \to P' = \Gr(m_{12}+m_{13}, \cE_1) \times_\cV
  \Gr(m_{22}+m_{12},\cE_2)$ is equal to
  \[ \groth_{(e_2)^{m_{11}}}(\cE_1/\cE'_1 - \cE_2)\, \groth_{R}(\cE'_1 -
     \cE'_2)\, \groth_{(m_{22}+m_{12})^{e_3}}(\cE_3 - \cE'_2)
  \]
  in
  $K_\circ(P')$.  After using the three-fold coproduct identity
  \[ \groth_{R}(\cE'_1 - \cE'_2) = \sum d^R_{\lambda,\mu,\nu}\,
     \groth_\lambda(\cE'_1 - \cE_2)\, \groth_\mu(\cE_2)\,
     \groth_\nu(-\cE'_2) \,,
  \]
  as well as the factorization identity
  \[ \groth_\nu(-\cE'_2)\, \groth_{(m_{22}+m_{12})^{e_3}}(\cE_3 -
     \cE'_2) = \groth_{(m_{22}+m_{12})^{e_3},\nu}(\cE_3 - \cE'_2) \,,
  \]
  it follows from Theorem~\ref{T:gysin} and Lemma~\ref{L:dualgysin}
  that the pushforward of the class in $K_\circ(P')$ along $P' \to
  \cV$ is equal to
\[ \pi_*([\cO_{\cV_{\bi,\br}}]) = \sum_{\lambda,\mu,\nu}
   d^R_{\lambda,\mu,\nu}\, 
   \groth_{(m_{22}+m_{23})^{m_{11}},\lambda}(\cE_1-\cE_2)\,
   \groth_\mu(\cE_2)\,
   \groth_{(m_{22}+m_{12})^{m_{33}},\nu}(\cE_3-\cE_2) \,,
\]
as required.
\end{proof}


\providecommand{\bysame}{\leavevmode\hbox to3em{\hrulefill}\thinspace}
\providecommand{\MR}{\relax\ifhmode\unskip\space\fi MR }
\providecommand{\MRhref}[2]{%
  \href{http://www.ams.org/mathscinet-getitem?mr=#1}{#2}
}
\providecommand{\href}[2]{#2}


\end{document}